\documentclass[pdflatex,sn-basic]{sn-jnl}

\usepackage{graphicx}%
\usepackage{multirow}%
\usepackage{amsmath,amssymb,amsfonts}%
\usepackage{amsthm}%
\usepackage{mathrsfs}%
\usepackage[title]{appendix}%
\usepackage{xcolor}%
\usepackage{textcomp}%
\usepackage{manyfoot}%
\usepackage{booktabs}%
\usepackage{algorithm}%
\usepackage{algorithmicx}%
\usepackage{algpseudocode}%
\usepackage{listings}%

\theoremstyle{thmstyleone}%
\newtheorem{theorem}{Theorem}
\newtheorem{lemma}[theorem]{Lemma}%

\theoremstyle{thmstyletwo}%
\newtheorem{remark}{Remark}%

\theoremstyle{thmstylethree}%

\newcommand{\bb}[1]{\boldsymbol{#1}}
\newcommand{\bm}[1]{\boldsymbol{#1}}
\newcommand{\R}{\mathbb{R}}
\newcommand{\ti}{\texttt{i}}

\newcommand{\tX}{\boldsymbol{X}}
\newcommand{\tY}{\boldsymbol{Y}}
\newcommand{\tG}{\boldsymbol{G}}
\newcommand{\E}{\mathsf{E}}
\newcommand{\var}{\mathsf{var}}
\newcommand{\tF}{\boldsymbol{F}}
\newcommand{\tZ}{\boldsymbol{Z}}
\newcommand{\tx}{\boldsymbol{x}}
\newcommand{\ty}{\boldsymbol{y}}
 \newcommand{\eps}{\varepsilon}
 
\newcommand{\ii}{\texttt{i}}

\newcommand{\hatt}{\widehat{\bb \vartheta}}
\usepackage{accents}

\begin{document}
 
\title[Optimal Transport--Based Multivariate Goodness-of-Fit Tests]{Optimal Transport--Based Multivariate Goodness-of-Fit Tests}

\author[1]{\fnm{Zden\v ek} \sur{Hl\'avka}}\email{hlavka@karlin.mff.cuni.cz}

\author*[1]{\fnm{\v{S}\'arka} \sur{Hudecov\'a}}\email{hudecova@karlin.mff.cuni.cz}
\equalcont{These authors contributed equally to this work.}

\author[2,3]{\fnm{Simos G. } \sur{Meintanis}} 
\equalcont{These authors contributed equally to this work.}

\affil*[1]{\orgdiv{Department
of Mathematical Statistics and Probability, Faculty of Mathematics and Physics}, \orgname{Charles University}, \city{Prague}, \country{Czech Republic}}

\affil[2]{\orgdiv{Department of Economics}, \orgname{National and Kapodistrian University of Athens},  \city{Athens},   \country{Greece}}

\affil[3]{\orgdiv{Pure and Applied Analytics}, \orgname{North--West University},   \city{Potchefstroom}, \country{South Africa}}

\abstract{Characteristic function-based goodness-of-fit tests are suggested for  multivariate distributions. The test statistics, which are straightforward to compute, are defined as two-sample criteria measuring discrepancy  of empirical characteristic functions between multivariate ranks of the original observations and the ranks obtained from an artificial sample generated from the reference distribution under test.
Multivariate ranks are constructed using the theory of the optimal measure transport, thus rendering the tests of a simple null hypothesis distribution-free, while bootstrap approximations are still necessary for testing composite null hypotheses. Asymptotic theory is developed and a simulation study, concentrating on comparisons with previously proposed tests of multivariate normality, demonstrates that the method performs well in finite samples. The broad applicability of the proposed methods is illustrated through an application to a real dataset.}


\maketitle


\section{Introduction}

Let  $\bb{X}\in \mathbb R^p, \ p\geq 1$, be a random vector with an absolutely continuous cumulative distribution function (DF) $F_{\bb{X}}$, given as $F_{\bb{X}}(\bb{x})=\mathbb P(\bb{X}\leq  \bb{x})$, $\bb{x}\in\mathbb{R}^p$. We shall begin our exposition with the {\it{simple}} hypothesis, whereby on the basis of independent copies ${\boldsymbol {\cal{X}}}_{n}:=\big\{\boldsymbol X_1,...,\boldsymbol X_n\big\}$ on $\bb X$, we consider the problem of goodness-of-fit (GoF) testing with the null hypothesis \begin{equation} \label{null} {\cal H}_0:
F_{\bb{X}}\equiv F_0, \end{equation} 
against general alternatives, where $F_0$ is the DF of a completely specified distribution. The discussion is extended to the problem of testing a \emph{composite} null hypothesis later in Section~\ref{sec_5}.

GoF tests for arbitrary distributions in the multivariate setting have become particularly important in recent times with the wide availability of data in high dimension. The case of normality is, not surprisingly, special, as there exists a plethora of GoF tests for the multivariate normal distribution; the reader is referred to the surveys of \cite{thode2002}, \cite{henze2002}, \cite{ebner2020}, \cite{arnast}, \cite{chen2023}.
On the other hand, outside the Gaussian context, there exist a few tests which however are only tailored for specific multivariate distributions \citep{meintanis2010, meintanis2015,meintanis2023},
while GoF methods for arbitrary multivariate laws are relatively scarce. A few exceptions are the general tests of 
\cite{gamero2009}, \cite{meintanis2014}, \cite{khmaladge2016}, \cite{ebner2018}, \cite{hallin2021_mordant}. 
Most of the aforementioned general methods however are computationally intensive, either in computing the test statistic itself and/or in getting critical values because of the need for bootstrap approximations. In other words an {\it{exactly}}, i.e. for each sample size,  distribution-free GoF test  rendering the bootstrap unnecessary, while at the same time involving a reasonable amount of computational complexity, is still missing in the literature.  The subject matter of this paper is to propose a computationally mild general approach for conducting GoF tests for multivariate laws that enjoys the property of exact distributional freeness.

To this end write $\varphi_{\bb X}(\boldsymbol t) =\mathsf E
[\exp({\texttt {i}}{\boldsymbol{t^\top X}})]$ and $\varphi_0(\bb t)$,  ${\boldsymbol{t}} \in \mathbb R^p$, for the  characteristic function (CF) corresponding to $F_{\bb X}$ and $F_0$, respectively. In view of the uniqueness of CFs, a reasonable CF--based criterion for the null hypothesis  ${\cal H}_0$ figuring in \eqref{null} is the statistic  
\begin{equation} \label{test}
D_n = D_{n}({\boldsymbol {\cal{X}}}_{n})=n  \int_{\mathbb R^p} |
\widehat{\varphi}_n({\boldsymbol{t}})-\varphi_0({\boldsymbol{t}})|^2 \ w({\boldsymbol{t}}) {\rm{d}}{\boldsymbol{t}},
 \end{equation}
where 
\begin{equation}\label{ecf}
\widehat{\varphi}_n({\boldsymbol{t}})=\frac{1}{n} \sum_{j=1}^{n}
\exp{({\texttt{i}} {\boldsymbol{t}}^\top {\boldsymbol{X}}_{j})},
\end{equation} 
is the empirical CF corresponding to ${\boldsymbol {\cal{X}}}_{n}$ and   $w(\cdot)$ is a nonnegative weight
function, satisfying
\begin{equation}  \label{weight}
   w({\boldsymbol{t}})=w(-{\boldsymbol{t}}),  \  
   \  \int_{\mathbb R^p}  w({\boldsymbol{t}}){\rm{d}}{\boldsymbol{t}}<\infty.
   \end{equation}
 Large values of $D_{n}$ indicate violation of the null hypothesis ${\cal{H}}_0$,  so the corresponding test rejects ${\cal{H}}_0$ if $D_n$ exceeds a certain threshold computed from the (asymptotic) distribution under the null. 

Goodness-of-fit criteria for multivariate observations incorporating test statistics as the one  in \eqref{test} may be computationally intensive, depending on the null CF $\varphi_0(\cdot)$, and have a typical but non-standard asymptotic  behavior under the null hypothesis ${\cal{H}}_0$ that involves the DF of $\boldsymbol X_1$. Our proposal herein also involves a CF--based test,  but in the interest of computational simplicity, rather than using the formulation figuring in \eqref{test}  
we employ the Monte Carlo approach of casting any GoF test as a two--sample test, suggested by \cite{chen2022}. 
Moreover distribution--freeness is achieved by employing, instead of the original observations, the corresponding multivariate ranks based on optimal transport theory recently developed by \cite{hallin2021} and \cite{deb2023}.

The rest of this work is outlined as follows.  Section \ref{sec_3} briefly introduces the concept of multivariate ranks. 
 In Section \ref{sec_2}, we define and compute the test statistic for a simple null hypotheses, while 
 the case of a composite null hypothesis involving unspecified distributional parameters is considered in Section \ref{sec_5}. 
The large--sample results under the null hypothesis as well as under alternatives are provided in Section \ref{sec_2}. 
 The finite-sample behavior of the methods is investigated by a Monte Carlo study 
 in Section \ref{sec_6}, and a real data application is provided in Section~\ref{sec:real}. The paper concludes with the discussion of results in Section \ref{sec:concl}. 
 All proofs are deferred to the Appendix, which also includes supplementary material related to the simulation study and the real data analysis.

\section{Multivariate ranks based on optimal transport}\label{sec_3}

 Univariate rank-based methods provide a robust, flexible, and distribution-free approach to statistical analysis. However, extending rank-based inference to a multivariate setting is not a straightforward task due to the absence of canonical ordering in $\R^p$ for $p>1$. Various concepts of multivariate ranks have been considered in the literature, for instance, componentwise ranks, spatial ranks and signs, depth-based ranks, or Mahalanobis ranks and signs, see \citet[Section~1.2]{hallin2021} for more details and further references.  
Recently, ranks and signs based on the optimal measure transport have become  popular and proved to be useful in various multivariate statistical  problems, see, e.g., \citet{shi, deb2023,hallin2023,huang2023,HH24}.

Let $(\mu,\nu)$ be a pair of probability measures on $\mathbb R^p$. 
The optimal measure transport (OMT)  problem, first formulated by \cite{monge}, 
 seeks to find a mapping $\bb{G}:\R^p\to\R^p$ that minimizes
 \begin{equation}\label{eq:monge}
\min_{\bb{G}}  \E \, \mathcal{C}\Big(\tZ,\bb{G}(\tZ)\Big)  {\ \rm{subject \ to\ }} \tZ\sim \mu, \  \bb{G}(\tZ)\sim \nu
\end{equation}
for a given cost function $\mathcal{C}:\R^p\times \R^p \to [0,\infty)$. We write $\bb{G}\#\mu=\nu$ when $\bb{G}(\boldsymbol Z)\sim \nu$ for $\boldsymbol Z\sim\mu$, and say that $\bb{G}$ pushes $\mu$ to $\nu$. 
The most common choice  is   $\mathcal{C}(\tx,\ty)=\|\tx-\ty\|^2$, therefore, we consider only this cost function in everything what follows. The concept of OMT has been recently used to define multivariate ranks and quantiles \citep{hallin2021, deb2023} and data depth \citep{galichon2017}.

Let $\tZ$ be a random vector with an absolutely continuous distribution $\mathcal{P}_{\tZ}$ on $\R^p$ and let $\nu$ be some specified reference measure on $\R^p$.   If $\tZ$ has finite second moments, then there exists
 a unique solution $\bb{G}^*$ to \eqref{eq:monge} with $\mathcal{C}(\tx,\ty)=\|\tx-\ty\|^2$. More generally (without the assumption $\E \|\tZ\|^2<\infty$), the mapping $\bb{G}^*$ can be defined as described in Remark~\ref{rem1} below. In any case, $\bb{G}^*$ is
in some sense unique mapping that pushes $\mathcal{P}_{\tZ}$ to $\nu$, see  Remark~\ref{rem1} for details.
\cite{deb2023} take $\nu$ as the uniform measure on $[0,1]^p$ and call $\bb{G}^*$ the  population rank function. 
\cite{hallin2021}  specify $\nu$ as the spherical distribution of a random vector
 $U \bm{S}$, 
where $U$ is  uniformly distributed on $[0,1]$ and  independent of $\bm{S}$ with a uniform distribution
on the unit sphere $ \mathcal{S}_{p}=\{\bm{x} \in \mathbb{R}^p; \|\bm{x}\| = 1\}$.
The optimal mapping $\bb{G}^*$ from \eqref{eq:monge} is then called the multivariate central-outward distribution function. 

Let  $\bm{\mathcal{Z}}_N: = \big\{\bb{Z}_1,\dots,\bb{Z}_N\big\}$ be a random sample from $\mathcal{P}_{\tZ}$. Moreover, let $\mathcal{G}_N =\{\bb{g}_1,\dots,\bb{g}_N\}$ be a given grid of points from the support of $\nu$. 
The empirical version $\widehat{\bb{G}}_N$ of  $\bb{G}^*$ is the solution to \eqref{eq:monge}  with $\mu$ and $\nu$ being the empirical distributions on  $\bm{\mathcal{Z}}_N$ and $\mathcal{G}_N$, respectively.  In other words, 
\begin{equation}\label{eq:monge_sample}
\widehat{\bb{G}}_N = \arg\min_{\bb{G}} \sum_{i=1}^N \|\bb{G}(\bb{Z}_i)-\bb{Z}_i\|^2,
\end{equation}
 where the $\arg\min$  is computed  among all bijections  $\bb{G}:\bm{\mathcal{Z}}_N \to \mathcal{G}_N$.
Consequently, $\widehat{\bb{G}}_N$ is used to define the multivariate ranks of $\bb{Z}_1,\dots,\bb{Z}_N$, but
the terminology here slightly varies: \cite{deb2023} call multivariate rank of $\tZ_i$ directly $\widehat{\bb{G}}_N(\tZ_i)$, while  \cite{hallin2021} define univariate ranks and multivariate signs   as simple functions of $\widehat{\bb{G}}_N(\tZ_i)$. 
 For simplicity, we call $\bb{R}_i= \widehat{\bb{G}}_N(\tZ_i)$ the multivariate rank of $\bb{Z}_i$ in what follows. Remark that computation of the discrete OMT $\widehat{\bb{G}}_N$  is a standard optimalization task that can be formulated as 
 a linear program for which efficient algorithms are available, see
\citet[Chapter~3]{peyre2019}.

 Notice that the multivariate ranks depend on the choice of the grid $\mathcal{G}_N$. Selection of the grid points should be driven by the choice of the theoretical reference measure $\nu$ in a sense that 
 the uniform measure on $\mathcal{G}_N$ converges weakly to $\nu$ as $N\to\infty$. In that case, it is possible to show that
 \[
 \frac{1}{N}\sum_{i=1}^N \| \bb{G}^*(\tZ_i) - \widehat{\bb{G}}_N(\tZ_i)\| \to 0 \quad \text{a.s. as } N \to \infty,
 \]
 see \citet[Theorem~2.1]{deb2023}. Under slightly stronger conditions, see \citet[Proposition~2.4]{hallin2021}, the convergence 
 is uniform 
 \begin{equation}\label{eq:GC}
 \max_{1\leq i\leq N} \| \bb{G}^*(\tZ_i) - \widehat{\bb{G}}_N(\tZ_i)\| \to 0 \quad \text{a.s. as } N \to \infty.
 \end{equation}
Due to their choice of $\nu$,  
 \cite{deb2023}  take the grid points $\mathcal{G}_N$ as a subset of $[0,1]^p$, while \cite{hallin2021} consider grid points in the unit ball in $\R^p$. In general, if the support of $\nu$ is $S\subset \R^p$, then the grid points should be distributed in $S$  
 so that the uniform measure on $\mathcal{G}_N$ well approximates $\nu$.
 There is not a unique approach to achieve this, and a construction of $\mathcal{G}_N$ often makes use of low-discrepancy sequences, e.g., Halton sequences \citep{halton} or GLP sets \citep[Section 1.3]{fang}.

If the reference measure $\nu$ is uniform on $[0,1]^p$ as in \cite{deb2023}, then the  grid set can be taken simply as a Halton sequence $\{\bb x_i\}_{i=1}^N$ in $[0,1]^p$ \citep{halton,randtoolbox}. Such grid will be referred to as a rectangular grid, abbreviated as $\mathcal{G}_N^R$. 

If $\nu$ is the distribution of the spherically uniform distribution of the random vector $U\bm{S}$, then \cite{hallin2021} recommend to take $\mathcal{G}_N$ as a collection of $n_0$ replications of $\bm 0$ and a set of points $\{\bb g_{ij}\}_{i=1,j=1}^{n_R,n_S}$, where $\bm{g}_{ij} =\frac{i}{n_R+1} \bm{s}_j$, where $n_0$, $n_R$, $n_S$ are such that $n_R\cdot n_S + n_0 = N$ and $\bm s_1,\dots,\bm s_{n_S}$ are directional vectors chosen uniformly as possible from the unit sphere $\mathcal{S}_p$. However, this approach requires the choice of $n_R$ and $n_S$. \cite{HHH2025} suggest to take $\{\bb g_i\}_{i=1}^N$ such that
\[
\bb g_i = x_{i,1}\cdot \bb{s}_i, 
\]
where $\{\bb x_i\}_{i=1}^N$ is a Halton sequence in $[0,1]^p$, 
$\bb x_i=(x_{i,1},\dots,x_{i,p})^\top$, so $x_{i,1}$ is the first element of $\bb x_i$, and $\bb s_i$ is a directional vector from the unit sphere in $\R^p$ computed as $\bb s_i = \tau(x_{i,2},\dots,x_{i,p})$, where $\tau$ is a mapping from $[0,1]^{p-1}$ to the unit sphere $\mathcal{S}_p$ such that $\tau(\bb U)$ has a uniform distribution on $\mathcal{S}_p$ whenever $\bb U$ has a uniform distribution on $[0,1]^{p-1}$, see \citet[Section 1.5.3]{fang}. Such grid will be referred to as a spherical grid, abbreviated as $\mathcal{G}_N^S$. 

A graphical illustration of $\mathcal{G}_N^R$ and $\mathcal{G}_N^S$ is provided in Figure~\ref{fig:grids} in Appendix~\ref{sec:grids}.

 \begin{remark}\label{rem1}
 \cite{hallin2021} define so-called center outward distribution function  $\tF_{\pm}$ more generally, without the assumption of finite second order moments of $\tZ$, 
  as the unique gradient of a convex mapping that 
 pushes $\mathcal{P}_{\tZ}$ to $\nu$, where $\nu$ is the distribution of the random vector $U \bm{S}$ specified above. The uniform convergence in \eqref{eq:GC} then holds for $\tG^*$ replaced by $\tF_{\pm}$.
  If $\E \|\tZ\|^2 < \infty$, 
  then it follows from Brenier’s theorem   \cite[Theorem 2.32]{villani} that $\tF_{\pm}=\bb{G}^*$, where $\bb{G}^*$ is the solution to the Monge's problem in \eqref{eq:monge}, see  \citet[Chapter~3]{villani}.  
  
  Hence, if one is not willing to make an assumption about finiteness of the second order moments of $\tZ$, then it is possible to define $\bb{G}^*$ as $\tF_{\pm}$. Such $\bb{G}^*$ still pushes $\mathcal{P}_{\tZ}$ to $\nu$ and enjoys a uniqueness property, since it is the unique gradient of a convex mapping with this property. Moreover, it is also the optimal mapping  with this property if  $\E \|\tZ\|^2 < \infty$.
\end{remark}

\section{Test statistic}\label{sec_2}

Consider a random sample ${\boldsymbol {\cal{X}}}_{n}$ from the absolutely continuous distribution $\mathcal{P}_{\tX}$ with DF $F_{\tX}$, and recall that the aim is to test $\mathcal{H}_0$ in \eqref{null} for a given absolutely continuous DF $F_0$. Let $m$ be some integer, the choice of which will be discussed later, 
 and let  ${\boldsymbol {\cal{X}}}^{(0)}_{m}:=\big\{\boldsymbol X^{(0)}_{1},...,\boldsymbol X^{(0)}_{m}\big\}$ be a random sample drawn from the distribution with DF $F_0$, independent of  ${\boldsymbol {\cal{X}}}_{n}$. Furthermore, let $\mathcal{G}_N$ be a specified grid of $N=n+m$ points. 

Denote as $\bm{\mathcal{Z}}_N = {\boldsymbol {\cal{X}}}_{n} \cup {\boldsymbol {\cal{X}}}^{(0)}_{m}$ the union of the two samples. 
For $\bm{\mathcal{Z}}_N$ and $\mathcal{G}_N$, we can compute the OMT $\widehat{\bb{G}}_N$ from \eqref{eq:monge_sample}. 
Let  ${\boldsymbol {\cal R}}_{n}:=\big\{\boldsymbol{R}_{i}\big\}_{i=1}^{n}$, $\bb{R}_i= \widehat{\bb{G}}_N(\tX_i)$, be the collection of rank vectors associated with ${\boldsymbol {\cal{X}}}_{n}$, and similarly ${\boldsymbol {\cal R}}^{(0)}_{m}:=\big\{\boldsymbol{R}^{(0)}_{j}\big\}_{j=1}^{m}$, $\bb{R}_j^{(0)} = \widehat{\bb{G}}_N(\tX_j^{(0)})$, be  the collection of rank vectors associated with ${\boldsymbol {\cal{X}}}^{(0)}_{m}$. 
Then a rank-based analogue of the test statistic $D_{n}({\boldsymbol {\cal{X}}}_{n})$ from \eqref{test} is given by  
\begin{equation} \label{ranktest}
D_{n,m}=D_{n,m}\big({\boldsymbol {\cal{R}}}_{n},{\boldsymbol {\cal{R}}}^{(0)}_{m}\big)=\frac{nm}{n+m}  \int_{\mathbb R^p} |
\widehat{\phi}_n({\boldsymbol{t}})-\widehat{\phi}^{(0)}_{m}({\boldsymbol{t}})|^2 \ w({\boldsymbol{t}}) {\rm{d}}{\boldsymbol{t}},
 \end{equation}
where $\widehat{\phi}_n$ and $\widehat{\phi}^{(0)}_{m}$ are computed as in \eqref{ecf} 
for ${\boldsymbol {\cal{R}}}_{n}$ and 
${\boldsymbol {\cal{R}}}^{(0)}_{m}$, respectively, that is
\[
\widehat{\phi}_n({\boldsymbol{t}})=\frac{1}{n} \sum_{j=1}^{n}
\exp{\big({\texttt{i}} {\boldsymbol{t}}^\top {\boldsymbol{R}}_{j}\big)}, \quad 
\widehat{\phi}_m^{(0)}({\boldsymbol{t}}) = \frac{1}{m} \sum_{j=1}^{m}
\exp{\big({\texttt{i}} {\boldsymbol{t}}^\top {\boldsymbol{R}}_{j}^{(0)}\big)}.
\]

The test statistic $D_{n,m}$ in \eqref{ranktest} is a two--sample test statistic that
measures the distance between the empirical CF of the ranks ${\boldsymbol {\cal{R}}}_{n}$ of the observed data ${\boldsymbol {\cal{X}}}_{n}$ and the corresponding empirical CF of the ranks ${\boldsymbol {\cal{R}}}^{(0)}_{m}$ obtained from an artificial random sample ${\boldsymbol {\cal{X}}}^{(0)}_{m}$ drawn from the reference distribution with DF $F_0$. 
The idea of casting any GoF test as a two--sample test goes back to \cite{friedman}, 
and has been used very effectively in \cite{chen2022} and \cite{karling2023}
for CF--based GoF testing, with promising results.
However while these works employ the original data here we propose to use the corresponding ranks.    As it will be discussed below, the use of appropriate multivariate ranks leads to an exactly distribution--free test while, additionally, basing the test on empirical CFs retains the computational simplicity and, of course, the consistency  of earlier CF--based tests.

Under the null hypothesis $\mathcal{H}_0$, both samples $\bb{\mathcal{X}}_n$ and $\bb{\mathcal{X}}^{(0)}_m$ come from the same distribution with DF $F_0$ and with a population OMT function $\bb{G}^*$ such that both $\bb{G}^*(\tX_i)$ and $\bb{G}^*(\tX_j^{(0)})$, $i=1,\dots,n$, $j=1,\dots,m$, have the reference distribution $\nu$. It follows from \eqref{eq:GC} that for $m,n$ large, both $\widehat{\phi}_n$ and $\widehat{\phi}_m^{(0)}$ will be close to the CF of $\nu$ and the test statistic $D_{n,m}$ is expected to be small. On the other hand, large values of $D_{n,m}$ indicate that  the distributions of  ${\boldsymbol {\cal{R}}}_{n}$ and  
${\boldsymbol {\cal{R}}}^{(0)}_{m}$ differ, indicating that the two samples ${\boldsymbol {\cal{X}}}_{n}$ and ${\boldsymbol {\cal{X}}}^{(0)}_{m}$ do not come from the same distribution. Therefore, the null hypothesis $\mathcal{H}_0$ is rejected if 
\[
D_{n,m}  > c_{n,m,\alpha},
\]
where $c_{n,m,\alpha}$ is a critical value such that the test keeps the prescribed level $\alpha$. Remark that $c_{n,m,\alpha}$ depends not only on the sample sizes $n,m$ and level $\alpha$, but also on the choice of  the grid points $\mathcal{G}_N$, the weight function $w$, and the dimension $p$. Note that the choice of the set $\mathcal{G}_N$ implicitly involves the choice of the reference measure $\nu$, i.e.~the type of the multivariate ranks.
On the other hand, the critical value $c_{n,m,\alpha}$ does not depend on the distribution $\mathcal{P}_{\tX}$, as justified in the following lemma.

\begin{lemma}\label{lem1}
Under the null hypothesis $\mathcal{H}_0$, the distribution of $D_{n,m}$ is free of $\mathcal{P}_{\tX}$. 
\end{lemma} 

The significance of the CF-based test statistic \eqref{test}, computed directly from the original data, is often evaluated via resampling or permutation techniques because its finite as well as asymptotic distribution typically depends on the unknown distribution $\mathcal{P}_{\tX}$ in a non-trivial way. In contrast to this,  
the finite sample distribution freeness of our test statistic $D_{n,m}$ allows to avoid these resampling techniques as the  critical value $c_{n,m,\alpha}$, for specified  $m,n$, $\mathcal{G}_N$ and  $w$, needs to be computed only once. Table~\ref{tab:crit} provides an example of critical values $c_{n,m,\alpha}$ for $p=2$.   Alternatively, the critical values for $D_{n,m}$ can be calculated from its asymptotic distribution provided in Section~\ref{sec_4}, but it is typically easier to compute $c_{n,m,\alpha}$ using Monte Carlo simulations.

\subsection{Computation and choice of weight function}\label{sec:comp}

Using the well-known routine calculations for the test statistic in \eqref{ranktest} with  a weight function satisfying \eqref{weight} yields the equivalent expression  
\begin{eqnarray} \label{sum}  
D_{n,m}
&=& \frac{m}{n(m+n)}\sum_{j,k=1}^n C_w(\boldsymbol R_{j}-\boldsymbol R_{k}) +\frac{n}{m(n+m)}\sum_{j,k=1}^{m} C_w(\boldsymbol R^{(0)}_{j}-\boldsymbol R^{(0)}_{k}) \notag \\ &&\mbox{}- \frac{2}{n+m}\sum_{j=1}^n \sum_{k=1}^{m} C_w(\boldsymbol R_{j}-\boldsymbol R^{(0)}_{k}),
\end{eqnarray}
where 
\begin{equation}\label{Cw}
C_w({\boldsymbol{x}})=\int_{\mathbb R^p} \cos({\boldsymbol{t}}^\top{\boldsymbol{x}})
 w({\boldsymbol{t}}){\rm{d}}{\boldsymbol{t}}.
 \end{equation}
If $w$ is selected as a density of a spherical distribution on $\R^p$,
then $C_w({\boldsymbol{x}})$ is its characteristic function, and 
the formula~\eqref{sum} provides a closed form expression for $D_{n,m}$ that avoids computation of multivariate integrals. Specifically, if $w$ is 
a density of a spherical stable distribution, then 
\begin{equation}\label{eq:Cw1}
C_w(\boldsymbol x)=\exp(- \|\boldsymbol x\|^\gamma), \ 0<\gamma\leq 2. 
\end{equation}
If $w$ is a density of a generalized spherical Laplace distribution, then $C_w(\boldsymbol x)=(1+\|\boldsymbol x\|^2)^{-\gamma}, \gamma>0$.   
Remark that in order to gain extra flexibility, one can compute $D_{n,m}$ from \eqref{sum} with $C_w$ enhanced by an additional scale parameter $a >0$. For instance, for the function from 
 \eqref{eq:Cw1} this leads to
\begin{equation}\label{eq:Cw}
C_{a,\gamma}(\bb x)=\exp(- \|a \bb x\|^\gamma),
\end{equation}
 which is for $\gamma=2$ the characteristic function of the normal distribution $\mathsf{N}(\bm{0}, 2a^2 \bm{I})$.

Remark that our test statistic is related to
the celebrated energy criterion of \cite{szekely2013}, 
  when applied to the ranks $\bb{\mathcal{R}}_n$ and $\bb{\mathcal{R}}_m^{(0)}$.
If $w$ is chosen such that $C_w$ is given as in \eqref{eq:Cw1} and one uses  the approximation ${\rm{e}}^x\approx 1+x$, then 
$D_{n,m}\approx  E_{n,m}$, where
\begin{align} \label{energy}
 E_{n,m}=&E_{n,m}({\boldsymbol {\cal{R}}}_{n},{\boldsymbol {\cal{R}}}^{(0)}_{m})= \frac{2}{n+m}\sum_{j=1}^n \sum_{k=1}^{m}\|\boldsymbol R_{j}-\boldsymbol R^{(0)}_{k}\|^\gamma \\  \notag
&-\frac{m}{n(n+m)}\sum_{j,k=1}^n \|\boldsymbol R_{j}-\boldsymbol R_{k}\|^\gamma -\frac{n}{m(n+m)}\sum_{j,k=1}^{m} \|\boldsymbol R^{(0)}_{j}-\boldsymbol R^{(0)}_{k}\|^\gamma.
\end{align}
It should be pointed out that the energy statistic {\it{based on the original observations}} is applicable only for $\gamma \in (0,2)$, and then only if $\mathsf E\|X_t\|^{\gamma}<\infty$, while the statistic in \eqref{test} or \eqref{ranktest} works for  $\gamma\in(0,2]$ (stable density)  or for all $\gamma>0$ (Laplace density), and is free of moment assumptions.  Because we use ranks though, the energy statistic  is expected to work even with heavy-tailed data.  

\cite{deb2023} proposed a rank energy statistic based on the OMT ranks (with a reference measure $\nu$ being uniform on $[0,1]^p$) for a two-sample problem for testing the equality of two continuous distributions. Their test statistic, for testing the equality of distributions of $\boldsymbol{\mathcal{X}}_n$ and $\boldsymbol{\mathcal{X}}^{(0)}_m$, takes the form
\eqref{energy} with $\gamma=1$. Hence, our procedure can be seen also as a generalization of the test recommended by \cite{deb2023} for a two-sample problem.

\begin{remark}
The idea of using an artificial sample to test some multivariate GoF hypothesis has been already used in \cite{chen2022}. Their simulations indicate that  efficiency can be improved by repeating the procedure $M$ times, for some chosen $M>0$, and using the average over these $M$ repetitions as the final test statistic. A similar idea could potentially be applied to the test based on $D_{n,m}$. However, this would require computing the discrete optimal transport in \eqref{eq:monge_sample} $M$-times. As a  result, the test would become computationally demanding for $n,m,N$ large.  Note also that the critical values would need to be recalculated accordingly.
\end{remark}

\subsection{Asymptotics   under the null hypothesis}\label{sec_4} 

In the following, we assume that  $\nu$ is a specified absolutely continuous reference measure on $\R^p$ with a compact support $\mathcal{S} \subset \R^p$,  
$N=n+m$, and the weight function $w$ satisfies~\eqref{weight}. 

\begin{theorem}\label{th2}
Let  $\{\mathcal{G}_N\}$ be a  sequence of grids in $\mathcal{S}$ such that the uniform measure on $\mathcal{G}_N$ converges weakly to $\nu$ as $n\to\infty$, and let $n/m\to\lambda \in(0,\infty)$ as $n\to\infty$. Then   under $\mathcal{H}_0$,
\begin{equation}\label{limit}
D_{n,m}\stackrel{\mathcal{D}}{\to} \int_{\R^p} Z^2(\bb{t})w(\bb{t}) \mathrm{d} \bb{t},
\end{equation}
as $n\to\infty$, where $\{Z(\bb t), \bb{t}\in\R^p\}$ is a centered Gaussian process with a covariance function 
\begin{equation}\label{eq:R}
R(\bb t_1,\bb t_2) = \E Z(\bb t_1)Z(\bb t_2) =  \mathsf{cov}\left(\cos(\bb t_1^\top \tY)+\sin(\bb t_1^\top \tY),\cos(\bb t_2^\top \tY)+\sin(\bb t_2^\top \tY)\right), 
\end{equation}
where $\tY$  is a random vector with distribution $\nu$. 
\end{theorem}

Theorem~\ref{th2} shows that the test based on $D_{n,m}$ is also asymptotically distribution-free, and the limiting distribution of $D_{n,m}$ is the same as the distribution of
$
D=\sum_{k=1}^{\infty} \lambda_k U_k^2$, where $U_k$ are iid standard normal variables and $\lambda_1\geq \lambda_2\geq \dots$ are 
 constants that depend on the weight function~$w$ and the reference measure $\nu$.
For a specified $w$, the asymptotic distribution from Theorem~\ref{th2} can be simulated and, subsequently, 
one can obtain asymptotic critical values $c_{\alpha}$, as described in Algorithm~\ref{alg0} in  the Appendix.
However, this procedure is numerically demanding for larger dimensions, so we rather recommend computing the critical values $c_{n,m,\alpha}$ by Monte Carlo simulations.

\begin{table}[htbp]
\scriptsize
\caption{Critical values $c_{\alpha}$ and $c_{n,m,\alpha}$ for level $\alpha=0.05$ dimension $p=2$, two types of grids  (rectangular $\mathcal{G}_N^R$ and spherical  $\mathcal{G}_N^S$), weighting by $C_{a,\gamma}$ in \eqref{eq:Cw} with $\gamma=2$, sample sizes $n$ and $m$. 
The asymptotic critical values $c_\alpha$ were computed using Algorithm~\ref{alg0} provided in Appendix~\ref{app:d} with $K=8$, $M=2\,000$, $G=8\,000$, and $B=40\,000$. The finite sample critical values $c_{n,m,\alpha}$ were
computed from $10\,000$ Monte Carlo simulations. 
}
\label{tab:crit}
\begin{tabular}{lc|ccccccc} 
\toprule
  & & & &$m=200$&&&$m=500$& \\ 
  \cmidrule(r){4-6}\cmidrule(r){7-9}
  & \multicolumn{1}{c|}{$a$} & Theorem~\ref{th2} & $n=20$&$n=50$&$n=80$&$n=20$&$n=50$&$n=80$ \\ 
  \cmidrule(r){1-9}
 &0.5 &  0.2270  &0.2262&0.2207&0.2240&0.2228&0.2255&0.2213 \\ 
 &1.0   &  0.7000  &0.6861&0.7020&0.6859&0.6850&0.6934&0.6929 \\ 
 $\mathcal{G}_N^R$&2.0 & 1.3542    &1.3130&1.3417&1.3496&1.3183&1.3458&1.3262 \\ 
 &3.0 &  1.4926  &1.4562&1.5000&1.4839&1.4573&1.4825&1.4745 \\ 
 &4.0 &  1.4899  &1.4401&1.4691&1.4555&1.4612&1.4839&1.4833 \\ 
 \cmidrule(r){2-9}
 &0.5 &   0.3936 &0.3773&0.3870&0.3815&0.3855&0.4079&0.3894 \\ 
 &1.0 &  0.9504  &0.9542&0.9518&0.9546&0.9338&0.9405&0.9536 \\ 
 $\mathcal{G}_N^S$&2.0 & 1.3844   &1.3784&1.3803&1.3794&1.3725&1.3717&1.3616 \\ 
 &3.0 & 1.4286   &1.4154&1.4237&1.4216&1.4335&1.4432&1.4416 \\ 
 &4.0 &  1.4059  &1.3831&1.3773&1.4013&1.4004&1.3719&1.3784 \\ 
\botrule
 \end{tabular}
\end{table}

Table~\ref{tab:crit} compares the asymptotic critical values $c_{\alpha}$ with the finite sample critical values $c_{n,m,\alpha}$ 
for dimension $p=2$, 
significance level $\alpha=0.05$, and various values of sample sizes $n,m$. 
The  weighting corresponds to $C_{a,\gamma}$ from \eqref{eq:Cw}  with $\gamma=2$.
It is visible that  $c_{n,m,\alpha}$ are close to $c_{\alpha}$ even for rather small values of $n$.

The condition $n/m \rightarrow \lambda \in (0, \infty)$  excludes non-standard cases in which the two sample sizes grow to infinity at different rates, while imposing no restriction on practical applications with a finite sample size $n$. The effect of the choice of $m$ on the power of the test is explored in the Monte Carlo simulation study in Section~\ref{sec_6}.

\subsection{Behavior under alternatives}

Recall that $\mathcal{P}_{\bb X}$ stands for the distribution of $\bb X_i$ and let   $\mathcal{P}_0$ be the distribution corresponding to $F_0$.  Consider now the alternative  $\mathcal{H}_1: F_{\bb X} \ne F_0$.
 In this case, $ \mathcal{P}_{\bb X} \ne \mathcal{P}_0$ and 
the pooled data $\bm{\mathcal{Z}}_N =\bm{\mathcal{X}}_n \cup \bm{\mathcal{X}}_m^{(0)}$ consist of two independent samples from two different distributions $\mathcal{P}_{\bb X}$ and $\mathcal{P}_0$.  The existence and uniqueness of the mapping $\widehat{\tG}_N$ are still guaranteed (by standard arguments).  Asymptotically,  $\widehat{\tG}_N$  behaves  as the solution to \eqref{eq:monge} for $\mu$ being a mixture of $\mathcal{P}_{\bb X}$  and $\mathcal{P}_{0}$. The next theorem describes the asymptotic behavior of the test statistic $D_{n,m}$ in this situation.

We still assume that  $\nu$ is absolutely continuous  with a compact support $\mathcal{S} \subset \R^p$,  
$N=n+m$,   and both $\mathcal{P}_{\bb X}$ and $\mathcal{P}_0$ are
absolutely continuous.

\begin{theorem}\label{th3}
Let both $\mathcal{P}_{\bb X}$ and $\mathcal{P}_0$  have finite second moments. Let $\{\mathcal{G}_N\}$  be a sequence of grids  on a compact set $S$ such that the uniform measure on $\{\mathcal{G}_N\}$  converges weakly to $\nu$ as $n\to \infty$ and $n/m\to\lambda \in (0,\infty)$. Let the weight function $w$ satisfy~\eqref{weight} and $\int_{\R^p} \|\bb t\| w(\bb t) \mathrm{d} \bb t <\infty$. 
Let $\bb G$ be the solution  to \eqref{eq:monge}  for $\mu = \frac{1}{1+\lambda} \mathcal{P}_{\bb X} + \frac{\lambda}{1+\lambda} \mathcal{P}_0$. 
Then as $n\to \infty$,
\begin{equation}\label{eq:alt}
\frac{D_{n,m}}{n+m} \stackrel{\mathsf{P}}{\to} \frac{\lambda}{(1+\lambda)^2}  \int_{\R^p}\left| \exp\big\{\ii\bb t^\top \bb G(\bb X_1)\big\} - \exp\big\{\ii \bb t^\top \bb G(\bb X_1^{(0)})\big\}\right|^2 w(\bb t) \mathrm{d} \bb t.
\end{equation}
 In particular, if $\mathcal{P}_{\boldsymbol{X}} \ne \mathcal{P}_0$ and $w$ is positive on an open neighborhood of $\bb 0$, then the right hand side of \eqref{eq:alt} is positive, and the test based on $D_{n,m}$ is consistent for $n \to \infty$.
\end{theorem}


The quantity on the right hand side of \eqref{eq:alt} is a positive multiple of the weighted $L_2$ comparison of the characteristic functions of 
$\bb G(\bb X_1)$  and $\bb G(\bb X_1^{(0)})$. 
 The assumptions of Theorem~\ref{th3} ensure that if $\mathcal{P}_{\bb X} \ne \mathcal{P}_0$, then the same holds for $\bb G(\bb X_1)$  and $\bb G(\bb X_1^{(0)})$. This phenomenon is illustrated by the empirical transform $\widehat{\bb G}_N$ in  Figure~\ref{fig:alt} in the Appendix~\ref{sec:grids}, which graphically compares 
 $\widehat{\bb G}_N(\bb X_i)$ and $\widehat{\bb G}_N(\bb X_j^{(0)})$ under various alternatives $\mathcal{P}_{\bb X} \ne \mathcal{P}_0$.


Note that it may be possible to weaken the assumption of finite second-order moments similarly to \cite{hallin2021}, but we do not explore this issue further.

\section{Test for a composite null hypothesis}\label{sec_5} 

This section addresses the problem of testing GoF to parametric families of distributions. Let 
$
\mathcal{F} = \{ F_{\boldsymbol\vartheta},  \boldsymbol\vartheta \in \Theta\}
$
be a family of distributions indexed by a   parameter $\bb \vartheta$ taking values in $\Theta\subseteq \mathbb R^q, \ q\geq 1$.
We wish to test the composite null hypothesis
\begin{equation}
 \label{null1} 
\mathcal{H}_0^C: F_{\boldsymbol X} \in \mathcal{F}
\end{equation} 
against a general alternative $\mathcal{H}_1^C: F_{\boldsymbol X} \not\in \mathcal{F}$.

The idea of the goodness-of-fit criterion constructed as a  two-sample comparison of the original data $\bm{\mathcal{X}}_n$ and an artificial sample can be extended to this composite null hypothesis $\mathcal{H}_0^C$. However, as the true value of the parameter $\bb \vartheta$ is unknown, the artificial sample is simulated from a distribution with parameter that is estimated from the data $\bm{\mathcal{X}}_n$. Such approach is typical in various
parametric bootstrap procedures.

Let $\widehat{\boldsymbol\vartheta}=\widehat{\boldsymbol\vartheta}_n$ be an estimator of $\bb{\vartheta}$ constructed from the sample ${\boldsymbol{\cal{X}}}_n$. 
Let ${\boldsymbol {\cal{X}}}_{m}^{(\widehat{\boldsymbol\vartheta})}:=\big\{\boldsymbol X_1^{(\widehat{\boldsymbol\vartheta})},\dots,\boldsymbol X_m^{(\widehat{\boldsymbol\vartheta})}\big\}$ be a 
Monte Carlo sample simulated from $F_{\widehat{\boldsymbol\vartheta}}$. 
We propose to evaluate the  composite hypothesis figuring in \eqref{null1} via a test statistic from \eqref{ranktest} computed for ${\boldsymbol{\cal{X}}}_n$ and  ${\boldsymbol {\cal{X}}}_{m}^{(\widehat{\boldsymbol\vartheta})}$.
Consider the pooled sample $ {\boldsymbol{\cal{X}}}_n \cup {\boldsymbol {\cal{X}}}_{m}^{(\widehat{\boldsymbol\vartheta})}$, and let ${\boldsymbol {\cal R}}_{n}$ and ${\boldsymbol {\cal R}}_{m}^{(\widehat{\boldsymbol\vartheta})}$  be the collection of multivariate ranks associated with observations in ${\boldsymbol {\cal{X}}}_{n}$ and  
${\boldsymbol {\cal{X}}}_{m}^{(\widehat{\boldsymbol\vartheta})}$, respectively, in this pooled sample. Define
 \begin{equation} \label{ranktest1}
\widetilde{D}_{n,m}=D_{n,m}({\boldsymbol {\cal{R}}}_{n},{\boldsymbol {\cal{R}}}^{(\widehat{\boldsymbol\vartheta})}_{m})=\frac{nm}{n+m}  \int_{\mathbb R^p} |
\widehat{\phi}_n({\boldsymbol{t}})-\widehat{\phi}^{(\widehat{\boldsymbol\vartheta})}_{m}({\boldsymbol{t}})|^2 \ w({\boldsymbol{t}}) {\rm{d}}{\boldsymbol{t}},
 \end{equation}
where $\widehat{\phi}^{(\widehat{\boldsymbol\vartheta})}_{m}$ is an empirical CF computed  from  ${\boldsymbol {\cal R}}_{m}^{(\widehat{\boldsymbol\vartheta})}$. 
In other words, $\widehat{\phi}^{(\widehat{\boldsymbol\vartheta})}_{m}(\cdot)$ is the analogue of $\widehat{\phi}^{(0)}_m(\cdot)$ figuring in \eqref{ranktest}, with extra randomness induced
by estimation of the parameter $\bm \vartheta$.
A closed-form expression, analogous to that in \eqref{sum}, is again available if $w$ is chosen appropriately.

As before, large values of  $\widetilde{D}_{n,m}$ indicate violation of the null hypothesis. An important difference is that the test based on $\widetilde{D}_{n,m}$ is no longer distribution free, because the distribution of the artificial sample $\bb{\mathcal{X}}_m^{(\hatt)}$ depends on the distribution of $\hatt$.    This issue is discussed in more detail in Remark~\ref{rem:s4} below.  Nevertheless, 
 the null hypothesis $\mathcal{H}_0^C$ is rejected if
\[
\widetilde{D}_{n,m}>\widetilde{c}_{n,m,\alpha},
\]
where $\widetilde{c}_{n,m,\alpha}$ is 
such that the test keeps the prescribed level $\alpha$. 
Namely, the critical value $\widetilde{c}_{n,m,\alpha}$ can be obtained using   the bootstrap procedure described in Algorithm~\ref{alg1}.

\begin{algorithm}[htbp]
\caption{Bootstrap procedure for calculation of $\widetilde{c}_{n,m,\alpha}$.}
\label{alg1}
\begin{algorithmic}[1]
\Require A sample $\boldsymbol{\mathcal{X}}_n=\{\tX_1, \dots, \tX_n\}$ and a family of distributions $\mathcal{F}$. Weight function $w$ or function $C_w$. 
\State compute the estimator  $\widehat{\boldsymbol\vartheta}$ under $\mathcal{H}_0^C$ from data $\boldsymbol{\mathcal{X}}_n$ 
\State generate $\boldsymbol {\cal{X}}_{m}^{(\widehat{\boldsymbol\vartheta})}$ from $F_{\widehat{\boldsymbol\vartheta}}$
\State  compute the value of the test statistic $\widetilde{D}_{n,m}$ 
\For{$b=1$ to $B$}
\State simulate data $\boldsymbol{\mathcal{X}}_n^b$ of size $n$ as iid from $F_{\widehat{\boldsymbol\vartheta}}$ 
\State  compute the estimate $\widehat{\boldsymbol\vartheta}^b$ from  $\boldsymbol{\mathcal{X}}_n^b$ 
\State  generate $\boldsymbol {\cal{X}}_{m}^{(\widehat{\boldsymbol\vartheta}^b,b)}$ from $F_{\widehat{\boldsymbol{\vartheta}}^b}$
\State compute
the corresponding bootstrap test statistic $\widetilde{D}_{n,m}^b=D_{n,m}({\boldsymbol {\cal{R}}}^{(\widehat{\boldsymbol\vartheta},b)}_{n},{\boldsymbol {\cal{R}}}^{(\widehat{\boldsymbol\vartheta}^b_n,b)}_{m} )$
\EndFor
\State set $\widetilde{c}_{n,m,\alpha}$ as $(1-\alpha)$-sample quantile of  $\widetilde{D}_{n,m}^1,\dots,\widetilde{D}_{n,m}^B$
\Ensure critical value $\widetilde{c}_{n,m,\alpha}$
\end{algorithmic}
\end{algorithm}

\begin{remark}\label{rem:s4}
The procedure is designed such that the artificial observations $\boldsymbol{\mathcal{X}}_{m}^{(\hatt)} =\{\boldsymbol{X}_1^{(\widehat{\boldsymbol\vartheta})},\dots,\boldsymbol X_m^{(\widehat{\boldsymbol\vartheta})}\}$ are mutually conditionally independent given $\widehat{\boldsymbol\vartheta}$. Unconditionally, they are dependent on the original data $\boldsymbol{X}_1,\dots,\boldsymbol{X}_n$, so  
generally unconditionally dependent. Furthermore, $\boldsymbol{X}_j^{(\hatt)}$ is distributed according to $F_{\widehat{\boldsymbol\vartheta}}$, given  $\widehat{\boldsymbol\vartheta}$, and so the unconditional distribution of $\bb X_j^{(\hatt)}$ is not $F_{\bb X}$. 
Hence, the pooled sample
\[
\bm{\mathcal{Z}}_N = \boldsymbol{\mathcal{X}}_n \cup \boldsymbol{\mathcal{X}}_m^{(\hatt)} = \{\bb X_1,\dots,\bb X_n, \bb X_1^{(\hatt)},\dots,\bb X_m^{(\hatt)}\}
\]
is a collection of variables that are not independent nor identically distributed. Moreover, the latter causes that the joint distribution of the pooled sample is not exchangeable. 
Consequently, the pooled multivariate ranks $\boldsymbol{\mathcal{R}}_n \cup \boldsymbol{\mathcal{{R}}}_m^{(\hatt)}$ are not exchangeable under the composite null hypothesis $\mathcal{H}_0^C$. This causes that the test based on $\widetilde{D}_{n,m}$ is not completely distribution free for finite $n,m$. The proof of Theorem~\ref{th2} relies on the permutation central limit theorem which is now not applicable due to the fact that the null distribution of  $\bb{\mathcal{R}}_n \cup \bb{\mathcal{{R}}}_m^{(\hatt)}$ is not exchangeable (i.e. permutation invariant).

If the estimator $\hatt$ is consistent, then under the null hypothesis, the pooled empirical measure on $\bb{\mathcal{ X}}_n \cup\bb{\mathcal{X}}_m^{(\hatt)}$ converges in probability to $F_{\bb X}$. This means that the pooled sample behaves asymptotically as a random sample from $F_{
\tX}$, and the empirical optimal transport $\widehat{\bb G}_N$ should be close to the true transport $\bb G^*$ of $F_{\tX}$ for large $n,m$. However, it is expected that the distribution of $\widetilde{D}_{n,m}$ is generally affected by the distribution of $\hatt$, as this is often the case in analogous goodness-of-fit procedures. A formal treatment of this would require limit theorems for the empirical transport of a collection of dependent and non-identically distributed data, that are, however, currently unavailable. Therefore, we leave the derivation of the asymptotic distribution of $\widetilde{D}_{n,m}$ as an open problem for future work.     

\end{remark}

\section{Testing an elliptical family} 

A special case of \eqref{null1} is obtained for $\mathcal{F}$ being a family of elliptical distributions on $\R^p$ with a specified radial distribution, see \citet[Section 1.2]{babic2021optimal} for an extensive list of related tests to this problem.  We say that a random vector $\bb X$ has an elliptical distribution $\mathcal{E}_p(\varphi,\bb \mu, \bb \Sigma)$ if it can be represented as 
\begin{equation}\label{eq:ellip}
\bb X = \bb \mu + R \bb A \bb S,
\end{equation}
where $\bb \mu \in \R^p$, $\bb A$ is a $p\times p$ matrix such that $\bb A \bb A^\top = \bb \Sigma$, $\bb S$ is a random vector uniformly distributed on the unit sphere in $\R^p$ and $R\geq 0$ is a radial random variable with density $\varphi$, independent of $\bb S$.    

Let $\varphi$ be specified. Then  the aim is to test that $F_{\tX}$ is a DF of $\mathcal{E}_p(\varphi,\bb \mu, \bb \Sigma)$  for some (unspecified) $\boldsymbol{\mu}$ and  $\boldsymbol{\Sigma}$.
A very important special case is obtained for $\varphi$ being the density of $\chi$ distribution with $p$ degrees of freedom, which leads to the class of $p$-variate normal distributions, 
and testing \eqref{null1} is testing normality. 

Following the procedure from Section~\ref{sec_5}, the artificial sample
 ${\boldsymbol {\cal{X}}}_{m}^{(\widehat{\boldsymbol\vartheta})}$ could be taken as a sample simulated from 
$\mathcal{E}_p(\varphi,\widehat{\boldsymbol{\mu}},\widehat{\boldsymbol{\Sigma}})$ for $\widehat{\boldsymbol{\mu}}$ and $\widehat{\boldsymbol{\Sigma}}$ being some suitable estimates of $\bb \mu$ and $\bb \Sigma$, respectively. 
The test based on $\widetilde{D}_{n,m}$ then proceeds as described above. 
For the special case of the normal distribution, 
one typically takes
$\widehat{\boldsymbol{\mu}}$ and $\widehat{\boldsymbol{\Sigma}}$ as the sample mean and sample covariance matrix of $\boldsymbol{\mathcal{X}}_n$, respectively, but some more robust estimators of $\boldsymbol{\mu}$ and $\boldsymbol{\Sigma}$ can also be considered. 


\medskip

However, for elliptical families, we propose also an alternative method 
to testing \eqref{null1} that removes the additional randomness in $\widetilde{D}_{n,m}$ arising from sampling the reference sample ${\boldsymbol {\cal{X}}}_{m}^{(\widehat{\boldsymbol\vartheta})}$.
Namely, 
instead of simulating independent random observations from $\mathcal{E}_p(\varphi,\widehat{\boldsymbol{\mu}},\widehat{\boldsymbol{\Sigma}})$, we propose to take the reference data as a grid that is fixed and non-random for given $\widehat{\boldsymbol{\mu}}$ and $\widehat{\boldsymbol{\Sigma}}$ and that corresponds to the structure of $\mathcal{E}_p(\varphi,\widehat{\boldsymbol{\mu}},\widehat{\boldsymbol{\Sigma}})$.
The construction of the grid is based on the representation \eqref{eq:ellip}. Let $H_{p}^{-1}$ be the quantile function of $R$ (which is known because the density $\varphi$ of $R$ is assumed to be known).  Let $U$ be uniformly distributed on $[0,1]$ and independent of $\bb S$ from \eqref{eq:ellip}. 
Then the random vector $\boldsymbol{Y} = U \boldsymbol{S}$ has spherically uniform distribution in the unit ball, and   $H_p^{-1}(U)$ has the same distribution as $R$, so
\[
\boldsymbol{Z} = \boldsymbol{\mu}+\boldsymbol{\Sigma}^{-1/2} H_p^{-1}(U)\boldsymbol{S}=
\boldsymbol{\mu}+\boldsymbol{\Sigma}^{-1/2} H_p^{-1}(\|\boldsymbol{Y}\|) \frac{\boldsymbol{Y}}{\|\boldsymbol{Y}\|}
\]
has $\mathcal{E}_p(\varphi,\bb \mu, \bb \Sigma)$ distribution. 

Let $\{\boldsymbol{y}_i\}_{i=1}^m$ be some fixed points from the unit ball such that 
the uniform distribution on  $\{\boldsymbol{y}_i\}_{i=1}^m$ is close to spherically uniform.
For instance, one can take $\{\bm{y}_i\}_{i=1}^m$ as the points from a spherical grid $\mathcal{G}_m^S$ described in Section~\ref{sec_3}.
Define
\begin{equation}\label{eq:xtilde}
\widetilde{\boldsymbol X}_i^{(\widehat{\boldsymbol\vartheta})} = \widehat{\boldsymbol{\mu}}+\widehat{\boldsymbol{\Sigma}}^{-1/2} H_p^{-1}(\|\boldsymbol{y}_i\|)\frac{\boldsymbol{y_i}}{\|\boldsymbol{y}_i\|}, \quad i=1,\dots,m.
\end{equation}
It follows from the above discussion that 
${\widetilde{\boldsymbol{\cal{X}}}}_{m}^{(\widehat{\boldsymbol\vartheta})}=\big\{\widetilde{\boldsymbol X}_1^{(\widehat{\boldsymbol\vartheta})},...,\widetilde{\boldsymbol X}_m^{(\widehat{\boldsymbol\vartheta})}\big\}$ should under the null hypothesis mimic a data set from $\mathcal{E}_p(\varphi,\bb \mu, \bb \Sigma)$. 
Subsequently, the test statistic in \eqref{ranktest1} can be computed for $\boldsymbol{\mathcal{X}}_n$ and ${\widetilde{\boldsymbol{\cal{X}}}}_{m}^{(\widehat{\boldsymbol\vartheta})}$ and their  OMT ranks $\bm{\mathcal{R}}_n$ and $\widetilde{\bm{\mathcal{R}}}_m^{(\hatt)}$. The resulting statistic is denoted as $\widetilde{D}_{n,m}^\star = D_{n,m}(\bm{\mathcal{R}}_n,\widetilde{\bm{\mathcal{R}}}_m^{(\hatt)})$. Large values of  $\widetilde{D}_{n,m}^\star$ indicate violation of the null hypothesis $\mathcal{H}_0^C$.

Even though the set ${\widetilde{\boldsymbol{\cal{X}}}}_{m}^{(\widehat{\boldsymbol\vartheta})}$ is non-random for given $\widehat{\bb \mu}$ and $\widehat{\bb \Sigma}$,   
 the randomness in the estimation of the parameters $\bb \mu$ and $\bb \Sigma$ and the consequent issues from Remark~\ref{rem:s4} remain present. Hence, the test statistic $\widetilde{D}_{n,m}^{\star}$ is not distribution free and
its significance needs to be evaluated using a  bootstrap procedure analogous to that in Algorithm~\ref{alg1}. 
 Nevertheless, compared to $\widetilde{D}_{n,m}$, the test statistic $\widetilde{D}_{n,m}^\star$ depends purely on the original data set $\boldsymbol{\cal X}_n$.  Consequently,  $\widetilde{D}_{n,m}^\star$ is fully reproducible and may be more appealing in practical applications. Moreover,
 the simulation study presented in Section~\ref{sec:composite} indicates that  the test based on $\widetilde{D}_{n,m}^\star$ even seems to be slightly more powerful compared to $\widetilde{D}_{n,m}$ for some settings.

 A graphical comparison of a random sample drawn from $\mathsf{N}_2(\widehat{\bm{\mu}},\widehat{\bm{\Sigma}})$ and a set of points obtained by \eqref{eq:xtilde} is provided in Figure~\ref{fig:grids2}
 in Appendix~\ref{sec:grids}.


\section{Simulation study}\label{sec_6}

The performance of the proposed GoF test is explored in a simulation study. Both types of problems, simple and composite null hypothesis, are considered. We focus on tests of multivariate normality in dimension $p\in\{2,3,4\}$, as this allows for a comparison between 
our approach and existing multivariate normality tests.   However, the optimal transport GoF test based on $D_{n,m}$ or $\widetilde{D}_{n,m}$ can be applied for testing any specified multivariate distribution in any dimension,  as illustrated later in Section~\ref{sec:real}.

For given sample sizes $n,m$, the computation of the test statistic requires a choice of the grid set $\mathcal{G}_N$ and the weight function $w$, or directly the function $C_w$ from \eqref{sum}. 
We present results for $C_w=C_{a,\gamma}$ from \eqref{eq:Cw} for $\gamma=2$ and various choices of $a>0$. The grid is taken either as rectangular $\mathcal{G}_N^R$ or spherical $\mathcal{G}_N^S$, where both types were introduced in Section~\ref{sec_3}.

\subsection{Simple null hypothesis}\label{sec:simple}

For a simple null hypothesis, the distribution function $F_0$ needs to be completely specified, and $D_{n,m}$ is distribution-free. 
The critical values $c_{n,m,\alpha}$ were computed for fixed $w$ and grid $\mathcal{G}_N$ only once from $10\,000$ independent replications. Remark that these Monte Carlo critical values are very close to the approximate asymptotic critical values, 
as already demonstrated in Table~\ref{tab:crit}.

We consider $F_0$ as the DF of the standard normal distribution $\mathsf{N}_p(\boldsymbol{0},\boldsymbol{I})$.  
Under the alternative, the data are simulated either from the uniform distribution on a specified set $A \subset \mathbb{R}^p$, denoted as $\mathsf{U}_A$, or from a $t$ distribution with 3 degrees of freedom. 
In the following, we use the notation $\mathsf{T}_p (\bm{\mu},\bm{\Sigma},d)$ for a $p$-variate $t$ distribution with $d$ degrees of freedom, location vector $\bm{\mu}$ and scale matrix $\bm{\Sigma}$.

 The obtained results for $p=3$ are summarized in Table~\ref{tab:single:p=3}, while results for $p=2$ and $p=4$ are provided in Tables~\ref{tab:single:p=2} and \ref{tab:single:p=4} in Appendix~\ref{app:d}.  
The empirical level and power are computed from 1\,000 simulations 
for the nominal significance level $\alpha=0.05$.
The results show that the empirical level of the test is close to the nominal value. Concerning the empirical power, the choice $a \in \{3,4\}$ can be recommended for the rectangular ranks (grid $\mathcal{G}_N^R$) and $a\in\{2,3\}$ seems to be reasonable for the spherical ranks (grid $\mathcal{G}_N^S$), but the optimal value of $a$ depends on the alternative. Clearly, the spherical ranks achieve larger power than the rectangular ranks for
all alternatives considered in Tables~\ref{tab:single:p=3}, \ref{tab:single:p=2}, and \ref{tab:single:p=4}. Therefore, we further focus solely on tests with spherical ranks $\mathcal{G}_N^S$ in the next section.

\begin{table}[htbp]
\scriptsize
\caption{Empirical level and power (in \%) for a goodness-of-fit test for the trivariate normal $\mathsf{N}_3(\boldsymbol{0},\boldsymbol{I})$ distribution (a simple hypothesis). 
$\mathsf{U}_{A}$ denotes the uniform distribution on the set $A$, $\mathsf{T}_3(\bm{0},\boldsymbol{I},3)$ is the trivariate $t$ distribution with 3 degrees of freedom, location $\bm{0}$ and identity scale matrix.
The values are computed from 1\,000 simulations for the nominal significance level $\alpha=0.05$.}
\label{tab:single:p=3}
\begin{tabular}{llr|rrr|rrr} 
\toprule
 & & & \multicolumn{3}{c}{$\mathcal{G}_N^R$}&\multicolumn{3}{c}{$\mathcal{G}_N^S$} \\ 
\cmidrule(r){4-6}\cmidrule(r){7-9}
 &\multicolumn{1}{c}{$m$} & \multicolumn{1}{c|}{$a$} & $n=20$&$n=50$&$n=80$&$n=20$&$n=50$&$n=80$ \\ 
\midrule
 &200&1.0&4.5&4.1&3.7&4.8&3.9&4.1 \\
 &&2.0&5.7&4.1&5.3&5.9&6.3&5.8 \\
 &&3.0&4.8&5.1&4.2&6.1&4.4&4.0 \\
 &&4.0&4.3&4.1&4.6&5.9&5.0&4.6 \\
 \cmidrule{3-3}\cmidrule{4-9} 
 \raisebox{1.5ex}[0pt]{$\mathsf{N}_3(\boldsymbol{0},\boldsymbol{I})$}&500&1.0&4.6&4.1&4.3&5.0&4.7&4.9 \\
 &&2.0&5.7&6.6&6.3&6.4&4.3&5.6 \\
 &&3.0&5.1&3.9&5.3&4.7&4.6&6.5 \\
 &&4.0&4.8&3.6&5.6&3.5&3.8&4.5 \\
 \cmidrule(r){1-3}\cmidrule{4-9}
 &200&1.0&12.1&13.8&21.1&38.2&64.6&81.9 \\
 &&2.0&22.5&45.0&68.0&59.3&90.8&97.9 \\
 &&3.0&22.3&55.0&74.6&58.7&93.9&98.8 \\
 &&4.0&17.4&53.0&76.1&56.5&92.0&97.9 \\
  \cmidrule{3-3}\cmidrule{4-9}
 \raisebox{1.5ex}[0pt]{$\mathsf{U}_{(-2,2)^3}$}&500&1.0&14.2&23.3&34.7&45.9&79.2&94.2 \\
 &&2.0&25.2&63.5&85.3&67.4&95.9&99.9 \\
 &&3.0&26.3&68.8&90.9&68.6&96.1&99.6 \\
 &&4.0&23.7&65.0&89.1&64.5&95.2&99.2 \\
  \cmidrule{1-3}\cmidrule{4-9} 
 &200&1.0&7.0&7.0&7.4&15.9&16.5&21.5 \\
 &&2.0&6.5&7.1&7.3&13.4&18.2&25.1 \\
 &&3.0&5.2&5.7&7.0&8.6&11.2&13.9 \\
 &&4.0&6.3&6.2&6.7&7.0&7.0&7.5 \\
  \cmidrule(r){3-3}\cmidrule{4-9}
 \raisebox{1.5ex}[0pt]{$\mathsf{T}_3(\bm{0},\boldsymbol{I},3)$}&500&1.0&7.1&8.1&5.8&16.3&26.4&34.9 \\
 &&2.0&6.7&10.4&10.5&10.9&23.5&29.6 \\
 &&3.0&5.5&7.9&11.1&10.3&10.5&17.9 \\
 &&4.0&4.1&5.7&9.7&6.8&7.0&10.7 \\
\botrule
\end{tabular}
\end{table}

\subsection{Composite null hypothesis}\label{sec:composite}

In order to investigate the test for a composite null hypothesis $\mathcal{H}_0^C$ in \eqref{null1},  we specify $\mathcal{F}$ as a family of  $p$-variate normal distributions.  
It follows from Section~\ref{sec_5} that for this case, the reference sample can be generated either as 
\begin{itemize}
\item[] (R) \quad  ${\boldsymbol {\cal{X}}}_{m}^{(\widehat{\boldsymbol\vartheta})}$ a sample from $\mathsf{N}_p(\widehat{\boldsymbol{\mu}}, \widehat{\boldsymbol{\Sigma}})$, or 
\item[] (G) \quad  
 ${\boldsymbol {\widetilde{\cal{X}}}}_{m}^{(\widehat{\boldsymbol\vartheta})}$ set of points defined in \eqref{eq:xtilde}  as 
appropriately shifted and rescaled spherically uniform grid points,
\end{itemize}
leading to test statistic $\widetilde{D}_{n,m}$ or $\widetilde{D}_{n,m}^\star$, respectively.

The empirical size of the test is computed for samples generated from normal distributions with various mean vectors and variance matrices, while the power is investigated for data simulated from the uniform distribution on  $[-1,1]^p$ and from the  centered t-distribution with 3 degrees of freedom and identity scale matrix. The dimension is chosen as $p\in\{2,3,4\}$. 
The obtained results are summarized in Tables \ref{tab:comp}--\ref{tab:warp}, and in Table \ref{tab:normality:p=34} in Appendix~\ref{app:d}.

\begin{table}[htbp]
\scriptsize
 \caption{Empirical level and power (in \%) for a test of bivariate normality (a composite hypothesis). Spherical ranks ($\mathcal{G}^S$), reference sample generated either as a random sample (R)  or computed as a 
 set of transformed grid points from \eqref{eq:xtilde} (G). $\mathsf{U}_{A}$ denotes the uniform distribution on the set $A$, $\mathsf{T}_2(\bm{0},\boldsymbol{I},3)$ is the bivariate $t$ distribution with 3 degrees of freedom, location $\bm{0}$ and identity scale matrix.  $\Sigma_{21}=\big(\begin{smallmatrix}
  2 & 1\\
  1 & 1
\end{smallmatrix}\big)$, $1\,000$ simulations.}
\label{tab:comp}

\begin{tabular}{rrr|rrr|rrr|rrr|rrr} 
\toprule 
 & & & \multicolumn{3}{c}{$\mathsf{N}_2(\boldsymbol{0},\boldsymbol{I})$}&\multicolumn{3}{c}{$\mathsf{N}_2 (\boldsymbol{1},\boldsymbol\Sigma_{21})$}&\multicolumn{3}{c}{$\mathsf{U}_{(-1,1)^2}$}&\multicolumn{3}{c}{$\mathsf{T}_2(\bm{0},\boldsymbol{I},3)$} \\[1ex] 
&&& \multicolumn{12}{c}{Sample size $n$}\\
\midrule
 &\multicolumn{1}{c}{$m$} & \multicolumn{1}{c|}{$a$} & $20$&$50$&$80$&$20$&$50$&$80$&$20$&$50$&$80$&$20$&$50$&$80$ \\ 
\cmidrule(r){1-3} \cmidrule(r){4-15}
 R&200&1.0&4.2&5.2&5.2&5.5&5.8&4.1&10.2&6.3&26.0&35.6&60.4&74.8 \\ 
 &&1.5&3.6&5.5&5.9&4.7&5.1&6.3&15.2&41.0&52.5&39.0&73.5&85.3 \\ 
 &&2.0&6.3&5.8&5.5&4.6&6.1&5.7&23.6&20.8&67.7&37.8&72.9&86.9 \\ 
 &&2.5&5.1&5.0&5.5&5.3&5.9&4.7&25.1&56.4&74.6&39.4&71.7&86.5 \\ 
 &&3.0&5.3&5.1&4.3&4.9&5.3&4.5&23.7&54.3&71.5&37.8&69.5&85.9 \\ 
 &&3.5&4.9&5.9&4.8&4.3&4.9&4.4&23.7&52.7&71.7&31.5&69.4&83.9 \\ 
\cmidrule(r){3-3}\cmidrule(r){4-15} 
&1\,000&1.0&5.1&5.6&6.3&4.5&5.2&5.9&9.6&49.8&71.6&49.7&84.5&94.6 \\ 
 &&1.5&3.7&5.9&5.5&5.9&4.0&5.3&21.0&65.8&90.5&48.7&84.4&95.1 \\ 
 &&2.0&5.1&4.4&5.1&6.3&5.0&5.8&27.8&71.6&93.5&46.7&81.3&93.8 \\ 
 &&2.5&6.6&4.4&5.8&5.5&5.0&5.9&27.8&74.3&91.2&42.5&81.0&94.7 \\ 
 &&3.0&5.3&5.9&4.2&6.2&6.2&4.8&28.1&70.3&88.8&38.5&78.7&92.1 \\ 
 &&3.5&4.7&5.3&6.0&4.3&5.7&4.6&27.1&68.7&89.9&39.0&77.2&91.9 \\ 
\midrule
 G&200&1.0&4.3&5.1&3.6&8.3&7.8&7.1&2.4&52.3&93.4&55.0&87.8&96.5 \\ 
 &&1.5&5.8&4.3&5.5&7.3&7.4&5.9&12.5&71.9&96.3&51.4&85.7&96.4 \\ 
 &&2.0&4.5&5.1&5.1&6.9&7.2&7.7&22.9&73.1&95.8&49.2&82.2&95.6 \\ 
 &&2.5&4.5&5.3&5.1&7.4&7.7&7.8&25.7&73.3&94.4&42.0&83.7&94.3 \\ 
 &&3.0&5.7&4.5&3.4&7.1&6.6&7.1&27.8&70.0&92.4&40.4&79.6&92.2 \\ 
 &&3.5&4.7&4.3&6.3&7.1&6.8&8.5&25.7&67.5&90.4&35.4&77.7&93.6 \\ 
 \cmidrule(r){3-3}
\cmidrule(r){3-3}\cmidrule(r){4-15}
 &1\,000&1.0&4.8&4.4&5.4&7.9&7.2&8.7&3.8&53.2&91.8&58.4&91.6&98.6 \\ 
 &&1.5&5.5&5.1&4.4&6.8&7.1&7.6&16.9&70.4&95.9&56.1&88.6&97.9 \\ 
 &&2.0&5.5&5.1&6.3&7.2&9.5&9.6&22.6&77.0&96.8&47.7&85.9&96.3 \\ 
 &&2.5&5.2&6.4&4.2&7.8&10.9&9.1&27.4&76.5&95.0&45.9&84.7&96.4 \\ 
 &&3.0&4.2&6.0&4.7&5.7&7.1&8.8&29.2&73.0&94.3&43.5&82.3&95.6 \\ 
 &&3.5&5.2&4.4&6.5&7.6&7.3&9.0&28.5&69.1&93.1&39.6&80.3&94.8 \\ 
\botrule
\end{tabular}

\end{table}


First, we aim to 
investigate differences between the two approaches (R) and (G) for generating the reference sample and differences between weight functions, represented by the  parameter $a$. 
For the composite hypothesis, it is necessary to use a resampling procedure to calculate  the critical values, but, unfortunately, running Monte Carlo simulations together with the bootstrap approximation described in Algorithm~\ref{alg1}  is prohibitively computationally expensive. At the same time, preliminary simulation results (not presented here) suggest that the dependency of the critical values 
$\widetilde{c}_{n,m,\alpha}$
on the vector $\boldsymbol{\vartheta}$ is not very strong,  and $\widetilde{c}_{n,m,\alpha}$ are close to the critical values obtained for the standard normal distribution provided that $\boldsymbol{\Sigma}$ is not  `extremely far' from the identity matrix. 
Therefore, to obtain the power results in Table~\ref{tab:comp} for dimension $p=2$, the bootstrap critical values were replaced by simpler critical values corresponding to the standard bivariate Normal distribution $\mathsf{N}_2(\boldsymbol{0},\boldsymbol{I})$, listed in Table~\ref{tab:crit:est} in Appendix~\ref{app:d}.  We believe that for this initial comparison of the tuning settings, this simplified 
procedure is satisfactory.  
Indeed, Table~\ref{tab:comp} suggests that this approach works quite well for randomly generated reference sample   in (R), whereas the empirical level obtained by using the grid points in (G)   seems to be somewhat inflated for the non-standard Normal distribution $\mathsf{N}_2 (\boldsymbol{1},\boldsymbol\Sigma_{21})$, where $\Sigma_{21}=\big(\begin{smallmatrix}
  2 & 1\\
  1 & 1
\end{smallmatrix}\big)$, so for (G), the critical values should be computed rather from the proper bootstrap, which is used in Tables~\ref{tab:warp} and \ref{tab:normality:p=34} for $p=2$ and $p\in\{3,4\}$, respectively.

Comparing the empirical power in  Table~\ref{tab:comp} observed for  $\mathsf{U}_{(-1,1)^2}$ and $\mathsf{T}_2(\bm{0},\boldsymbol{I},3)$ distributions, we conclude that a larger reference sample size ($m=1\,000$) is to be recommended. For this choice, the powers for the two approaches (R) and (G) are similar. 
The choice of the tuning parameter $a$ seems to be crucial for the $\mathsf{U}_{(-1,1)^2}$ alternative where, for instance, for the random grid with $m=200$ and $n=80$, the power ranges from 8.0 \% ($a=0.5$) to 71.7~\% ($a=3.5$). For the $\mathsf{T}_2(\bm{0},\boldsymbol{I},3)$ alternative, the differences between powers for different $a$ are smaller. In general, $a$ from the interval $[2,3]$ seems to be a reasonable choice for practical usage.

Finally, Table~\ref{tab:warp} and Table~\ref{tab:normality:p=34} in Appendix \ref{app:d}  compare the test based on $\widetilde{D}_{n,m}^\star$ (with a grid type reference sample) to a set of standard multivariate normality goodness-of-fit tests for dimensions $p=2$ and $p\in\{3,4\}$, respectively.
The results for  $\widetilde{D}_{n,m}^\star$ are computed from the proper bootstrap-based estimates of the critical values $\widetilde{c}_{n,m,\alpha}$.
Since the full bootstrap is computationally too demanding,  the empirical level and power were computed using the so-called warp-speed method, see \citet{warp}.
The reference normality tests are computed using  R library \texttt{MVN} \citep{MVN}. Namely, we present results of  Mardia's multivariate skewness  and kurtosis coefficients, denoted as M3 and M4, respectively, see \cite{mardia1974applications},  Henze-Zirkler's (abbreviated as H-Z) multivariate normality test from \cite{HZ}, Royston's multivariate test of \cite{royston}, Doornik-Hansen's (abbreviated as D-H) multivariate normality test from \cite{Doornik}, and  Energy multivariate normality test of \cite{szekely2013}.

 It is discussed in Section~\ref{sec:comp} that the test statistic ${D}_{n,m}$ is, for some specific weight functions, approximately related to the two sample energy statistic for the OMT ranks, considered for a comparison of two independent samples and rectangular grid in \cite{deb2023}.  Hence, 
Table~\ref{tab:warp} provides also power results for a test 
based on $\widetilde{E}_{n,m}^\star = E_{n,m}(\bb{\mathcal{R}}_n, \widetilde{\bb{\mathcal{R}}}_m^{(\widehat{\bb \vartheta})})$, where $E_{n,m}$ is defined by \eqref{energy} for $\gamma=1$, denoted in the table as D-S. The significance was computed analogously as for the test based on $\widetilde{D}_{n,m}^\star$, using $m=1\,000$. In contrast to \cite{deb2023}, we use spherical ranks that provided better power in  Section~\ref{sec:simple}. The results indicate that the power of the D-S test is comparable to $\widetilde{D}_{n,m}^\star$ with the same sample sizes $n,m$.

 \begin{table}[htbp]
 \scriptsize
 \caption{Empirical level and power (in \%) for a goodness-of-fit test of bivariate normality (a composite hypothesis). 
Spherical ranks ($\mathcal{G}^S$),  reference sample considered as a 
 set of transformed grid points from \eqref{eq:xtilde} (approach (G)). $\mathsf{U}_{A}$ denotes the uniform distribution on the set $A$, $\mathsf{T}_2(\bm{0},\boldsymbol{I},3)$ is the bivariate $t$ distribution with 3 degrees of freedom, location $\bm{0}$ and identity scale matrix. Matrices $\Sigma_{21}=\big(\begin{smallmatrix}
  2 & 1\\
  1 & 1
\end{smallmatrix}\big)$,  $\Sigma_{10,3}=\big(\begin{smallmatrix}
  10 & 3\\
  3 & 1
\end{smallmatrix}\big)$. 
D-S denotes a test that uses the two-sample energy test statistic from \citet{deb2023} with $m=1\,000$ and otherwise the same setup.
The other reference tests are  computed using R library \texttt{MVN}. All quantities are obtained from $1\,000$ simulations. }
\label{tab:warp}
\begin{tabular}{rc|rrr|rrr|rrr|rrr} 
\toprule
  & &\multicolumn{3}{c}{$\mathsf{N}_2 (\boldsymbol{1},\boldsymbol\Sigma_{21})$}&\multicolumn{3}{c}{$\mathsf{N}_2 (\boldsymbol{1},\boldsymbol\Sigma_{10,3})$}&\multicolumn{3}{c}{$\mathsf{U}_{(-1,1)^2}$}&\multicolumn{3}{c}{$\mathsf{T}_2(\bm{0},\boldsymbol{I},3)$} \\[1ex] 
  && \multicolumn{12}{c}{Sample size $n$}\\
   \cmidrule(r){1-14} 
$m$ & $a$ & $20$&$50$&$80$&$20$&$50$&$80$&$20$&$50$&$80$&$20$&$50$&$80$ \\ 
 \cmidrule(r){1-2} \cmidrule(r){3-14} 
&2.0&3.0&5.8&5.8&4.7&5.0&5.4&15.9&75.1&96.3&44.8&85.7&95.5 \\ 
 200&2.5&5.4&4.2&4.9&5.7&7.9&3.7&27.4&73.4&94.3&40.5&80.7&92.8 \\ 
 &3.0&4.5&3.3&5.4&4.5&4.3&3.2&22.9&68.5&94.6&38.0&80.7&92.3 \\ 
  \cmidrule(r){2-2}\cmidrule(r){3-14} 
 &2.0&7.0&4.0&5.6&4.6&4.9&5.0&23.4&77.1&96.1&47.4&85.4&96.3 \\ 
 1\,000&2.5&4.4&4.9&7.3&3.3&5.5&6.0&28.2&78.9&95.4&43.0&83.9&95.7 \\ 
 &3.0&4.2&5.5&5.0&4.3&5.6&4.4&27.9&75.2&93.6&39.4&82.1&95.1 \\ 
  \cmidrule(r){1-2}\cmidrule(r){3-14} 
\multicolumn{2}{l|}{D-S}&6.1&4.3&6.3&5.0&4.8&6.1&27.3&73.6&94.6&48.2&86.9& 97.5\\
  \cmidrule(r){1-2}\cmidrule(r){3-14} 
  \multicolumn{2}{l|}{H-Z}   & 5.0 & 5.5 & 4.7 &5.0&5.5&4.7& 16.8 & 69.8 &  93.4 & 43.2 & 80.4 & 92.3   \\
 \multicolumn{2}{l|}{M3}       & 3.6 & 5.0 & 4.3 &3.6&5.0&4.3&  0.1 &  0.1 &   0.0 & 44.7 & 77.7 & 85.5  \\
 \multicolumn{2}{l|}{M4}       & 0.7 & 1.8 & 2.2 &0.7&1.8&2.2&  0.2 & 63.9 &  97.1 & 34.3 & 88.4 & 97.5  \\
 \multicolumn{2}{l|}{Royston}          &7.0&7.6&5.8& 6.3&6.7&5.6&27.5 & 94.9 & 100.0 & 55.0 & 87.7 & 96.7   \\
 \multicolumn{2}{l|}{D-H}  &4.9&4.9&5.7& 4.2&4.1&4.8& 7.5 & 68.2 &  97.4 & 44.0 & 85.6 & 95.8  \\
 \multicolumn{2}{l|}{energy}           &6.1&6.1&4.4&6.1 &6.1&4.4& 10.0 & 55.8 &  88.0 & 48.8 & 84.9 & 95.4 \\
\botrule
 \end{tabular}
 \end{table}

 Table~\ref{tab:warp} for $p=2$ shows that the empirical level of most of the considered tests is close to the nominal value $\alpha=0.05$, with lower values observed for the M4 test and slightly higher values for the Royston test.   Under the considered alternatives, 
the  results for the candidate tests are rather comparable for the $\mathsf{T}_2(\bm{0},\boldsymbol{I},3)$ alternative, while more visible differences are observed for the uniform alternative $\mathsf{U}_{(-1,1)^2}$. The results indicate 
that the Royston test yields the largest empirical power in most of the cases. Our approach with $a=2$ has a similar power as the the Doornik-Hansen and energy tests for the  $\mathsf{T}_2(\bm{0},\boldsymbol{I},3)$ distribution.
For the uniform $\mathsf{U}_{(-1,1)^2}$ distribution, 
 the M3 test completely fails to detect the alternative. The remaining tests
  (including our approach) have similar power for $n=80$ observations. The Royston test clearly outperforms all other tests for $n=50$ and it achieves similar power as our approach (with $a\in\{2.5,3\}$) for $n=20$. 
     Similar conclusions can be made for dimensions $p=3$ and $p=4$  based on results from Table~\ref{tab:normality:p=34} in Appendix~\ref{app:d}.


Altogether, our approach seems to be similarly powerful as the standard 
multivariate normality tests and, in terms of empirical power, it is significantly outperformed only by the (slightly oversized) Royston test. 


\section{Real-data application}\label{sec:real}

 Understanding the joint distribution of socioeconomic variables is a central problem in econometrics. In particular, the relationship between income and expenditures plays a fundamental role in analyzing economic behavior and inequality. Examining the joint distribution of these two variables allows us to capture important features such as nonlinear dependence, asymmetry, and tail behavior, which are not revealed by simple correlation analysis or linear regression models.

 As an example, we analyze a two-dimensional data set consisting of expenditures per student
 and district average income (both in USD 1,000) in Californian schools \citep{SW2007} in 1998--1999, plotted in Figure~\ref{fig:CAS-plot}. The data are available in the R library \texttt{AER} \citep{AER}.
For illustrative purposes, we consider both the complete data set and a subset consisting of the first 100 observations. As can be seen from Figure~\ref{fig:CAS-plot}, for a sample size of 100, the heavy tails of the distribution are not yet fully apparent, whereas they become clearly visible in the plot of the complete data set. The first 100 observations tend to come from  lower income districts and, therefore, they may not capture all features present in the entire data.
Consequently, one may expect the goodness-of-fit test to yield different results in the two cases. 

Figure~\ref{fig:CAS-hist} in Appendix~\ref{app:d} shows histograms of the two variables. They suggest that the marginal distribution of income exhibits positive skewness relative to the normal distribution, while the skewness of expenditures is more modest.   This feature is already apparent even for the small sample size $n=100$.

\begin{figure}
\begin{center}
\includegraphics[width=0.45\textwidth]{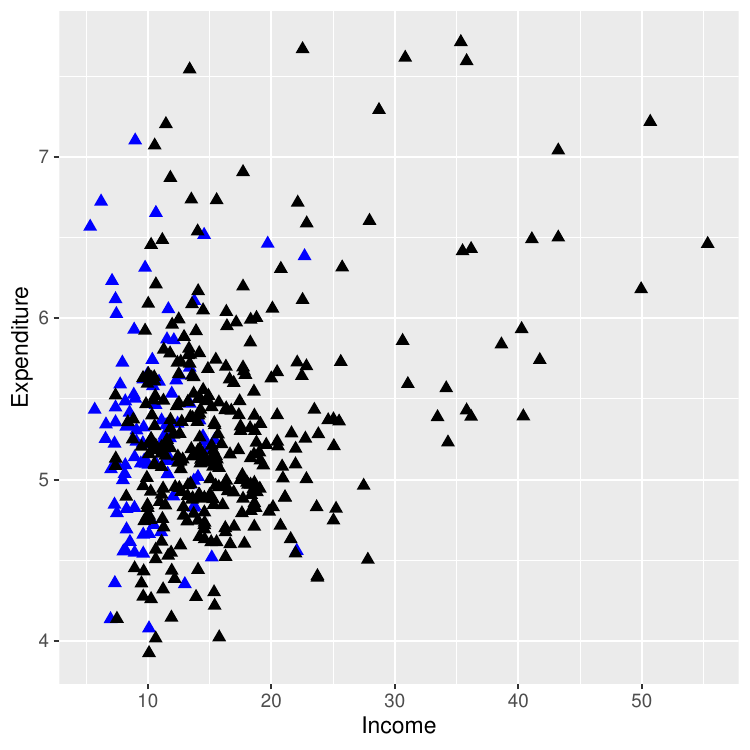}
\includegraphics[width=0.45\textwidth]{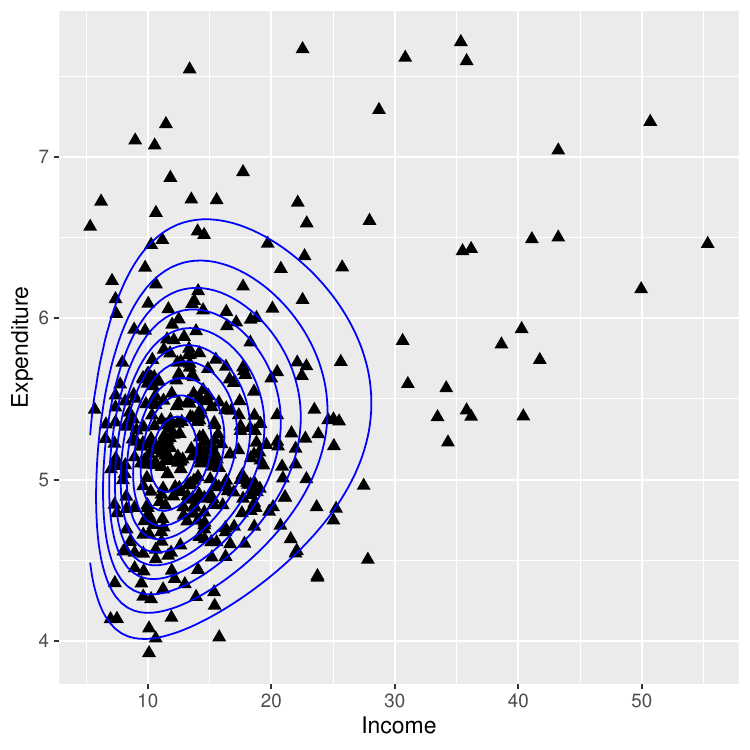}
\end{center}
\caption{The bivariate real data on income and expenditure from the California schools (left panel). The first $n=100$ observations, considered also in the analysis, are stressed by blue color. All the data together with contours of the fitted distribution with lognormal marginals and normal copula (right panel).}\label{fig:CAS-plot}
\end{figure}

 To get flexible models for the data, we define the tested class of distributions $\mathcal{F}$ by specifying parametric models for the two marginals and a parametric copula. The marginals are considered as normal, lognormal, and Gamma,  while the copula is chosen from the normal, Frank or Gumbel families. Since it is reasonable to assume the same marginal model for both variables, this results in nine candidate models for the bivariate distribution.

For each specified $\mathcal{F}$, the goodness-of-fit test is conducted as described in Algorithm~\ref{alg1}. In Steps~1 and~6, the parameter vector $\bb \vartheta$ is estimated by the  two-step Inference Functions for Margins (IFM) method \citep{joe2005}. More precisely, the marginal distribution parameters are estimated first  by maximum likelihood using the function \texttt{fitdistr()} from the R package \texttt{MASS} \citep{MASS}. In the second step, the copula parameter is estimated from the resulting pseudo-observations using the function \texttt{fitCopula()} from the R package \texttt{copula} \citep{copula1,copula2,copula3}.
 The artificial data are generated from the fitted distribution, in steps 2, 5, and 7, by standard procedures.  
Figure~\ref{fig:CAS-plot} (right panel) compares the data with contours of 
the distribution with lognormal marginals and a normal copula, fitted by this procedure.

 The tuning parameters of the GoF test are chosen based on their good performance in the simulation study. In particular,  $m=200$, the grid is spherical and the weight parameter is $a=2.5$, with $B=1\,000$ bootstrap replications.   As mentioned above,  for illustrative purposes, we present results not only for the full data set of $n=420$ observations, but also for a subset of the first $n=100$ records. 

The aim is to compare the results of the test in these two cases, since features such as heavy tails and extreme observations are less pronounced in the smaller sample, whereas they become more evident when the entire data set is considered. Comparing the two cases also illustrates the behavior of the goodness-of-fit test with respect to the amount of available data.
 The resulting p-values are summarized in Table~\ref{CAS-table}. 

\begin{table}
\scriptsize
\caption{P-values for goodness-of-fit tests of the nine candidate models specified by the three marginal distributions (normal, lognormal, Gamma) and three copula families: normal, Frank, and Gumbel. The test is applied to the subset of the first $n=100$ observations as well as to the whole data set ($n=420$). 
The 
GoF test parameters are set as $m=200$, $a=2.5$,  spherical grid, $B=1\,000$ bootstrap replications.}\label{CAS-table}
\begin{tabular}{r|ccc|ccc}
\toprule
  & \multicolumn{3}{c|}{$n=100$} & \multicolumn{3}{c}{$n=420$} \\ 
 & Normal & Frank  & Gumbel  & Normal  & Frank  & Gumbel  \\
\midrule
Normal     & 0.004 & 0.063 & 0.042 & 0.000 & 0.000 & 0.000 \\
Lognormal & 0.580 & 0.078 & 0.773 & 0.515 & 0.068 & 0.028 \\
Gamma     & 0.438 & 0.163 & 0.431 & 0.002 & 0.000 & 0.000 \\
\bottomrule
\end{tabular}
\end{table}

For the first $n=100$ observations, all the three  bivariate distributions with normal marginals 
provide rather a poor fit. This is consistent with the conclusions drawn from  histograms in Figure~\ref{fig:CAS-hist}.  In particular, 
the null hypothesis is rejected 
 at the test level $\alpha=0.05$ for the normal copula ($p$-value 0.004), where  in combination with normal marginals, $\mathcal{F}$ is the class of bivariate normal distributions. 
 A borderline $p$-value is obtained for the remaining two copula families combined with normal marginals. 
On the other hand, lognormal and Gamma marginals 
lead
to an acceptable fit for the first $n=100$ observations when combined with all three copulas,  
with p-values between 0.078 (lognormal marginals with Frank copula) to 0.936 (lognormal marginals with Frank copula).

When testing the entire dataset of $n=420$ observations, the model with lognormal marginals and normal copula seems to provide the best fit, see also the right panel of Figure~\ref{fig:CAS-plot}. A bivariate distribution with lognormal marginals and Frank copula would still be acceptable, while the other models are rejected at level $\alpha=0.05$. 

Based on the results, the distribution with lognormal marginals and a normal copula seems to be an appropriate model for the considered data set. The results showed that with only 100 observations, the goodness-of-fit test was not powerful enough to distinguish between lognormal and Gamma marginal distributions, or between normal, Frank, and Gumbel copulas. This illustrates that the power of the goodness-of-fit test may deteriorate in smaller samples, where important distributional features, such as heavy tails and extreme observations, are not yet fully manifested.

\section{Conclusion}\label{sec:concl}

This paper deals with goodness-of-fit testing for a specified multivariate distribution, and proposes a test statistic that makes use of multivariate ranks derived from the optimal measure transport theory. 
We show that the test of a simple null hypothesis is distribution-free, while a composite null hypothesis  requires a bootstrap approximation of the critical values.

The empirical results presented in Section~\ref{sec_6} show that the proposed test achieves power comparable to that of the most powerful existing normality tests. However, the optimal transport-based goodness-of-fit test is significantly more general, as it can be applied to testing arbitrary multivariate distributions with only minor modifications. This flexibility was demonstrated in Section~\ref{sec:real} for a real data example utilizing separate models for marginal distributions and a copula. 
 The test requires only a procedure for parameter estimation and a method for generating samples from the distribution under the null hypothesis.

\subsection*{Acknowledgments}
The research of Šárka Hudecová and Zdeněk Hlávka was supported by the Czech Science Foundation project GA\v{C}R No.~25-15844S. We would like to also thank the Editor and the Reviewers for their helpful comments and suggestions. Their feedback helped us improve the paper and clarify several parts of the manuscript.

The paper was completed after the sad passing of Simos Meintanis. His contribution was essential, and the remaining authors are sincerely grateful to have had the privilege of collaborating with him.

\section*{Declarations}
{\small
\noindent \textbf{Competing interests:} The authors declare that they have no competing interests.\\
\noindent \textbf{Data availability statement:} The data used in Section~\ref{sec:real} are publicly available through the R package \texttt{AER} from CRAN, \cite{AER}.  
}

\appendix

\section{Grids}\label{sec:grids}

An example of the two types of  grids described in Section~\ref{sec_3}, rectangular $\mathcal{G}_N^R$ and spherical   $\mathcal{G}_N^S$, is provided in Figure~\ref{fig:grids} for $N=500$ and $p=2$.
For dimension $p=2$, the grid points $\bm{g}_i=(g_{i,1},g_{i,2})^\top$ of the spherical grid   $\mathcal{G}_N^S$ are computed simply as
\begin{equation}\label{grid:2d}
g_{i,1} = x_{i,1} \cos(2\pi x_{i,2}), \quad g_{i,2} = x_{i,1} \sin(2\pi x_{i,2}), \quad i=1,\dots,N,
\end{equation}
where $\{\bm{x}_{i}\}_{i=1}^N$ is a Halton sequence in $[0,1]^2$.

Under $\mathcal{H}_0$, the vector of pooled multivariate ranks $\bb{\mathcal{R}}_n \cup \bb{\mathcal{R}}_m^{(0)}$ is uniformly distribution on the set of all permutations of the grid points $\mathcal{G}_N$. On the other hand, Figure~\ref{fig:alt} demonstrates that if the two samples $\bb{\mathcal{X}}_n$ and $\bb{\mathcal{X}}_m^{(0)}$ come from two different distributions, then the distributions of $\bb R_i$ and $\bb R_j^{(0)}$ differ. If the distributions $F_{\bb X}$ and $F_0$ differ in location, then the ranks $\bb R_i$, $i=1,\dots,n$, accumulate in a different part of the grid than $\bb R_j^{(0)}$, $j=1,\dots,m$, so the empirical distributions of $\bb{\mathcal{ R}}_n$ and  $\bb{\mathcal{ R}}_n^{(0)}$ differ in location as well.
Similarly, if $F_{\bb X}$ and $F_0$ differ in scale, then the same holds for the the empirical distributions of $\bb{\mathcal{R}}_n$ and  $\bb{\mathcal{ R}}_n^{(0)}$ , as it is visible from the middle panel of Figure~\ref{fig:alt}. Finally, even if the location and scale of $F_{\bb X}$ and $F_0$  are the same, but one of them is symmetric, while the second one is skewed, then one can clearly see in the right panel of Figure~\ref{fig:alt} that  $\bb R_i$, $i=1,\dots,n$ gather in different parts of the grid than $\bb R_j^{(0)}$, $j=1,\dots,m$.

\medskip

Figure~\ref{fig:grids2} compares a random sample $\bm{\mathcal{X}}_m$ of size $m$ drawn from $\mathsf{N}_2(\bm{\mu},\bm{\Sigma})$   with the set  $\widetilde{\bm{\mathcal{X}}}_m = \{\widetilde{\boldsymbol X}_i \}_{i=1}^m$ where 
$\widetilde{\boldsymbol X}_i $ is computed   as

\begin{equation}\label{eq:xtilde2}
\widetilde{\boldsymbol X}_i = \bm{\mu}+ \boldsymbol{\Sigma}^{-1/2}   H_2^{-1}(\|\boldsymbol{y}_i\|)\frac{\boldsymbol{y_i}}{\|\boldsymbol{y}_i\|}, \quad i=1,\dots,m,
\end{equation}
where $H_2$ is DF of  
a $\chi_2$ distribution  (i.e. the square root of a $\chi^2$ distribution with 2 degrees of freedom) and $\{\bm{y}_i\}_{i=1}^m$ are from $\mathcal{G}_m^S$   from  \eqref{grid:2d}.
 It is visible that the empirical distribution of  $\widetilde{\bm{\mathcal{X}}}_m$ mimics $\mathsf{N}_2(\bm{\mu},\bm{\Sigma})$, but the points are distributed more  regularly than for the random sample $\bm{\mathcal{X}}_m$.

\begin{figure}[htbp]
\begin{center}
\includegraphics[width=0.8\textwidth]{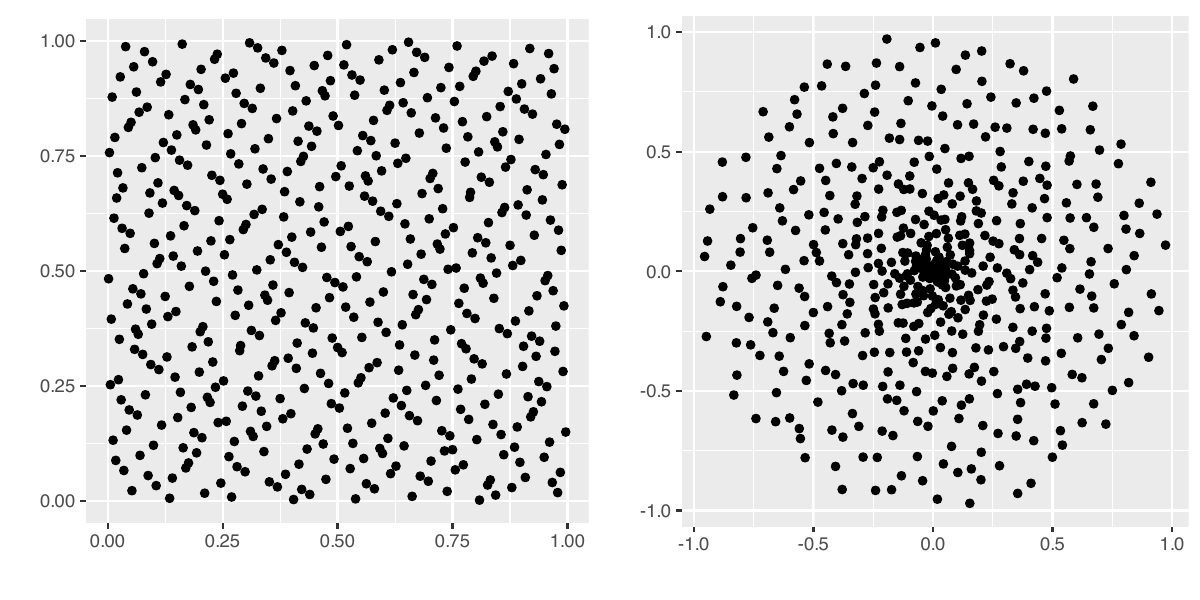}
\end{center}
\caption{A rectangular grid $\mathcal{G}_N^R$ (left panel) and a spherical grid  $\mathcal{G}_N^S$ (right panel) for $N=500$ and $p=2$.}\label{fig:grids}
\end{figure}

\begin{figure}[htbp]
\begin{center}
\includegraphics[width=\textwidth]{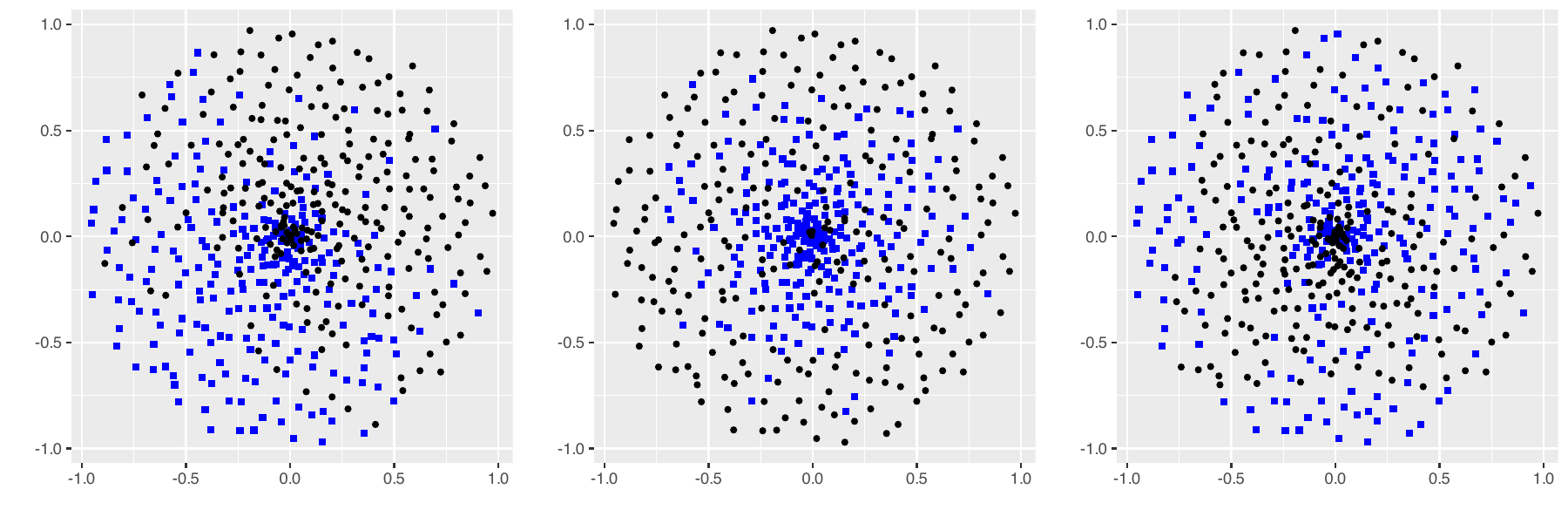}
\end{center}
\caption{Example of differences between $\bb R_i = \widehat{\bb G}_N(\bb X_i)$ (blue squares) and $\bb R_j^{(0)} = \widehat{\bb G}_N(\bb X_j^{(0)})$ (black circles) when $\bb X_i$ and $\bb X_j^{(0)}$ come from two different distributions $F_{\bb X} \ne F_0$ differing in  
 location (left panel) or in scale (middle panel). The right panel shows results for $F_{\bb X}$ being spherically symmetric, while $F_0$ is skewed.  In all cases $n = m = 250$, so $N=500$. 
}\label{fig:alt}
\end{figure}

\begin{figure}[htbp]
\begin{center}
\includegraphics[width=0.8\textwidth]{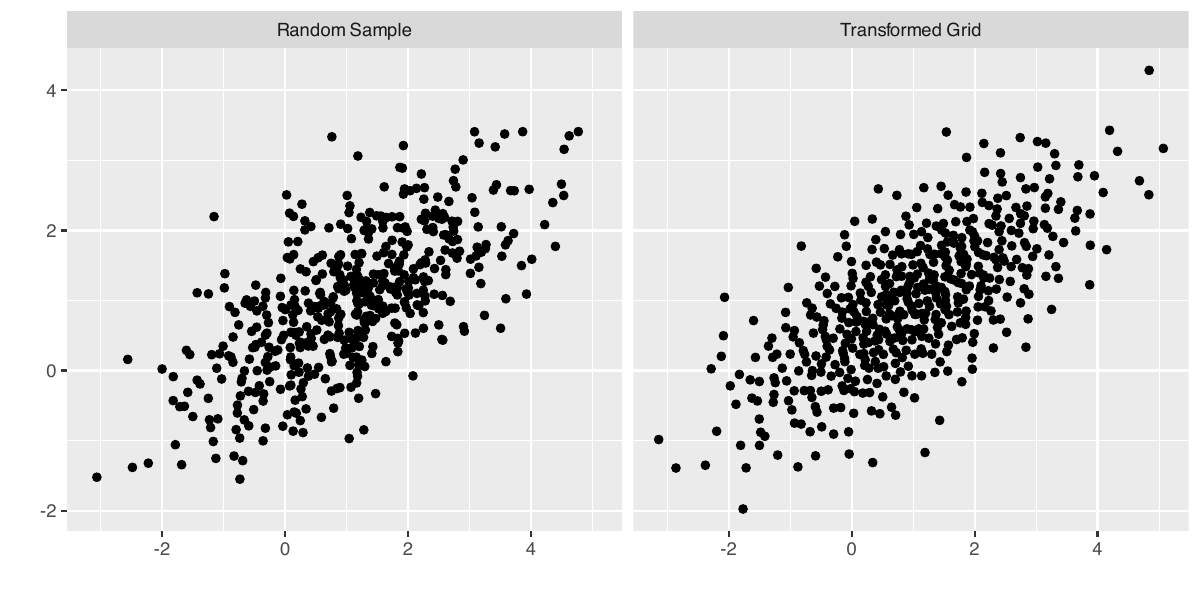}
\end{center}
\caption{A  random sample from $\mathsf{N}_2(\bm{\mu},\bm{\Sigma})$ of size $m$ (left panel) and a set of $m$ transformed grid points from \eqref{eq:xtilde2}
(right panel) for $m=500$, $\bm{\mu}=(1,1)^\top$,  $\Sigma=\big(\begin{smallmatrix}
  2 & 1\\
  1 & 1
\end{smallmatrix}\big)$.
}\label{fig:grids2}
\end{figure}

\section{Proofs}\label{sec:ap1}

\begin{lemma}\label{lem2}
Let $f:\{1,\dots,N\}\to\mathbb C$ be a function and $\pi=(\pi_1,\dots,\pi_N)$ be a permutation of $(1,\dots,N)$. Then for $1<n<N$, 
\[
\frac{1}{n}\sum_{i=1}^{n}  f({\pi_i}) - \frac{1}{N-n}\sum_{i=n+1}^N f({\pi_i}) = \frac{N}{n (N-n)} \left[ \sum_{i=1}^n f(\pi_i) - \frac{n}{N}\sum_{i=1}^N f(i)\right].
\]
\end{lemma}
\begin{proof}
See that
\[
\frac{1}{N-n}\sum_{i=n+1}^{N}  f({\pi_i}) = \frac{1}{N-n}\left[\sum_{i=1}^{N}  f({\pi_i}) - \sum_{i=1}^n  f({\pi_i})\right] =  \frac{1}{N-n}\left[\sum_{i=1}^{N}  f({i}) - \sum_{i=1}^n  f({\pi_i})\right].
\]
The claimed expression is obtained by direct calculations. 
\end{proof}

\begin{proof}[Proof of Lemma~\ref{lem1}.]
Let $\mathcal{G}_N = \{\bb{g}_j\}_{j=1}^N$. Denote as 
\[
\phi^{\bb g}(\bb t) = \frac{1}{N}\sum_{j=1}^N \exp(\ti \bb{t}^\top \bb{g}_j). 
\]
It follows from Lemma~\ref{lem2} that
\[
\widehat{\phi}_n(\bb t) - \widehat{\phi}_m^{(0)}(\bb t) = \frac{N}{m}\big[\widehat{\phi}_n(\bb t)-\phi^{\bb g}(\bb t)\big].
\]
Hence, for a fixed $\mathcal{G}_N$ and $w$, the only source of randomness 
in $D_{n,m}$ comes from $\bb{\mathcal{R}}_n$.
Under  $\mathcal{H}_0$, the pooled sample $\bm{\mathcal{Z}}_N = {\boldsymbol {\cal{X}}}_{n} \cup {\boldsymbol {\cal{X}}}^{(0)}_{m}$  is a random sample from $\mathcal{P}_{\tX}$ and therefore, the vector of all ranks $\Big(\widehat{\tG}(\tX_1),\dots,\widehat{\tG}(\tX_n),\widehat{\tG}(\tX_1^{(0)}),\dots, \widehat{\tG}(\tX_m^{(0)})\Big)^\top$ has a uniform distribution over the $N!$ permutations of the grid set $\mathcal{G}_N$, see \citet[Proposition 2.2]{deb2023} or \citet[Proposition 2.5]{hallin2021}. Therefore, $\bb{\mathcal{R}}_n$ is uniformly distributed on the set of all subsets of $\mathcal{G}_N$ of size $n$. This implies that 
\[
\mathsf{P}\left(\widehat{\phi}_n(\bb t) = \frac{1}{n}\sum_{j=1}^n \exp({\ti \bb{t}^\top \bb{g}_{\pi_j}})\right) = \frac{1}{{N \choose n}}
\]
for any permutation $(\pi_1,\dots,\pi_N)$ of $(1,\dots,N)$. Since $D_{n,m}$ is only a transformation of $\widehat{\phi}_n(\bb t)$, it follows that its distribution is the same for all $\mathcal{P}_{\tX}$.
\end{proof}


Recall that  $\nu$ is a specified absolutely continuous reference measure on $\R^p$ with a compact support $\mathcal{S} \subset \R^p$, 
  $N=n+m$, and  \eqref{weight} holds.

\begin{lemma}\label{lem3}
Let $\pi=(\pi_1,\dots,\pi_N)$ be a random permutation of $(1,\dots,N)$ and let  $\{\mathcal{G}_N\}$, $\mathcal{G}_N = \{\bb g_i\}_{i=1}^N$, be a sequence of grids such that the uniform measure on $\mathcal{G}_N$ converges weakly to $\nu$ as $N\to\infty$. Let $F\subset \R^p$ be a compact. 
Set
\[
S_N(\bb t) =  \sum_{i=1}^N d_{\pi_i} \bigl[\cos(\bb{t}^\top\bb{g}_{i})+\sin(\bb{t}^\top\bb{g}_{i})\bigr],
\]
where
\[
d_j = \begin{cases} \sqrt{\frac{N-n}{n N}},& i\leq n,\\
-\sqrt{\frac{n}{(N-n)N}},& i>n.
\end{cases}
\]
If $n/(N-n) \to \lambda\in(0,\infty)$ as $N\to \infty$, then
\begin{equation}\label{eq:F}
\int_F S_{N}^2(\bb t) w(\bb t) \mathrm{d} \bb t \stackrel{D}{\to} \int_F Z(\bb t)^2  w(\bb t) \mathrm{d} \bb t, 
\end{equation}
where $Z$  is a centered Gaussian process with covariance function in \eqref{eq:R}. 
\end{lemma}

\begin{proof}
 Notice that it follows from the assumptions that $n\to\infty$ if and only if $N\to\infty$. Hence, all convergences below hold for $n\to\infty$ as well as $N\to \infty$. 
Set  $a_{\bb t}(i) = \cos(\bb t^\top \bb g_i)+\sin(\bb t^\top \bb g_i)$, then 
$
 S_N(\bb t)= \sum_{i=1}^N d_{\pi_i} a_{\bb t}(i).
$
Recall that $m=N-n$. See that
\begin{align*}
\overline{d}_N &:= \frac{1}{N}\sum_{i=1}^N d_i = \frac{1}{N} \sqrt{\frac{N}{nm}} \left[n \frac{m}{N} - m \frac{n}{N}\right]=0,\\
\sigma^2_d &:= \frac{1}{N}\sum_{i=1}^N (d_i - \overline{d}_N)^2  = \frac{1}{N}\frac{N}{nm}\left[ n \left(\frac{m}{N}\right)^2 + m \left(\frac{n}{N}\right)^2\right] = \frac{1}{N},\\
\max_{1\leq i\leq N} (d_i - \overline{d}_N)^2 &= \frac{1}{N}\max\left\{\frac{n}{m},\frac{m}{n}\right\}.
\end{align*}
Let $\tY$ be a random vector with distribution $\nu$ and define $W(\bb t) = \cos(\bb t^\top \tY)+\sin(\bb t^\top \tY)$. 
Since the uniform measure on $\mathcal{G}_N$ converges weakly to $\nu$, we have
\[
\frac{1}{N} \sum_{i=1}^N h(\bb g_i) \to \int_{\R^p} h(\tx) \mathrm{d}\nu(\tx) = \E h(\tY)
\]
for any continuous bounded function $h$. Therefore, 
\begin{align*}
\overline{a}_{\bb t,N} &:= \frac{1}{N}\sum_{i=1}^N a_{\bb t}(i) \to \E W(\bb t), \\
 \sigma^2_{a,\bb t} &:=\frac{1}{N}\sum_{i=1}^N[a_{\bb t}(i) - \overline{a}_{\bb t,N}]^2 \to \var [ W(\bb t)], 
\end{align*}
 and
\[
\max_{1\leq i\leq N}[a_{\bb t}(i) - \overline{a}_{\bb t,N}]^2 \leq \max_{1\leq i\leq N} 2[a_{\bb t}(i)^2 + \overline{a}_{\bb t,N}^2]\leq 2(4+4)=16.
\]
Therefore,
$
\E S_N(\bb{t}) = N \overline{a}_{\bb t,N} \overline{d}_N = 0
$
and
\begin{equation}\label{eq:varS}
\var S_N(\bb t) = \frac{N^2}{N-1} \sigma^2_{a,\bb t} \sigma^2_d  \to \var \, W(\bb t).
\end{equation}
Since
\begin{align*}
N &\frac{\max_{1\leq i\leq N}[a_{\bb t}(i) - \overline{a}_{\bb t,N}]^2}{\sum_{i=1}^N [a_{\bb t}(i) - \overline{a}_{\bb t, N}]^2}  \frac{\max_{1\leq i\leq N}[d_i - \overline{d}_N]^2}{\sum_{i=1}^N [d_i - \overline{d}_N]^2}\\
& \quad  \quad  \leq  \frac{16}{ \frac{1}{N}\sum_{i=1}^N[a_{\bb t}(i) - \overline{a}_{\bb t,N}]^2} \frac{1}{N} \max\left\{\frac{n}{m},\frac{m}{n}\right\} \to 0,
\end{align*}
it follows from Hoeffding's combinatorial central limit theorem, \cite[Theorem~4]{hoef}, that $S_N(\bb t) \stackrel{D}{\to} \mathsf{N}\big(0, \var\,  W(\bb t) \big)$.

Let $K>1$, $\bb t_1,\dots,\bb t_K$ be from $\R^p$ and let $\lambda_1,\dots,\lambda_K \in \R$. Consider
\[
\sum_{j=1}^K \lambda_j S_N(\bb t_j)  = \sum_{i=1}^N d_{\pi_i} \underbrace{\sum_{j=1}^K \lambda_j a_{\bb t_j}(i)}
_{b(i)} =  \sum_{i=1}^N d_{\pi_i}b(i).
\]
Then
\[
\overline{b}_N := \frac{1}{N} \sum_{i=1}^N b (i ) = \sum_{j=1}^K \lambda_j \overline{a}_{\bb t_j, N}
\]
and
\begin{align*}
\sigma^2_b &:= \frac{1}{N}\sum_{i=1}^N [b(i)-\overline{b}_N]^2 = \frac{1}{N}\sum_{i=1}^N \left[\sum_{j=1}^K \lambda_j [a_{\bb t_j}(i) - \overline{a}_{\bb t_j,N}]\right]^2 \\
&=\sum_{j=1}^K \sum_{l=1}^K \lambda_j \lambda_l \frac{1}{N} \sum_{i=1}^N  [a_{\bb t_j}(i) - \overline{a}_{\bb t_j,N}][a_{\bb t_l}(i) - \overline{a}_{\bb t_l,N}]\\
&\to \sum_{j=1}^K \sum_{l=1}^K\lambda_j \lambda_l  \mathsf{cov}\big(W(\bb t_j),W(\bb t_l)\big) =  \sum_{j=1}^K \sum_{l=1}^K\lambda_j \lambda_l  R(\bb t_j,\bb t_l). 
\end{align*}
It follows from the Cauchy-Schwarz inequality that 
\begin{align*}
 [b(i) -\overline{b}_N]^2 \leq  16 K \sum_{j=1}^K \lambda_j^2,
\end{align*}
and therefore,
\[
N \frac{\max_{1\leq i\leq N}[b(i) - \overline{b}_{N}]^2}{\sum_{i=1}^N [b(i) - \overline{b}_{N}]^2} \frac{\max_{1\leq i\leq N}[d_i - \overline{d}_N]^2}{\sum_{i=1}^N [d_i - \overline{d}_N]^2}\to 0,
\]
 and the Hoeffding's combinatorial central limit theorem implies that
 \[
 \sum_{j=1}^K \lambda_j S_N(\bb t_j)  \stackrel{D}{\to} \mathsf{N}\left(0,  \sum_{j=1}^K \sum_{l=1}^K\lambda_j \lambda_l  R(\bb t_j,\bb t_l)\right).
 \]
 This proves convergence of finite dimensional distributions of $S_{N}$ to finite dimensional distributions of $Z$.  

To prove \eqref{eq:F}, it remains to  show that
\begin{equation}\label{eq:C1}
\sup_N \E \int_F S_N(\bb t)^2 w(\bb t) \mathrm{d} \bb t <\infty
\end{equation}
and there exist $C>0$ and $\kappa>0$ such that
\begin{equation}\label{eq:C2}
\sup_N \E |S_N^2(\bb t_1) - S_N^2(\bb t_2)|\leq C \| \bb t_1-\bb t_2\|^\kappa,
\end{equation}
see \citet[Theorem 22]{ibragimov}.  It follows from \eqref{eq:varS} that
\[
\E S_N(\bb t)^2 = \var S_N(\bb t) = \frac{N}{N-1} \sigma^2_{a,\bb t} \leq M_1
\]
for some real constant $M_1>0$. Hence, \eqref{eq:C1} follows from integrability of $w$. Furthermore, the Cauchy-Schwarz inequality yields
\begin{align*}
\E |S_N^2(\bb t_1) - S_N^2(\bb t_2)&|\leq \sqrt{\E [S_N(\bb t_1) - S_N(\bb t_2)]^2} \sqrt{\E [S_N(\bb t_1) + S_N(\bb t_2)]^2} \\
&= \sqrt{\var Q_N(\bb t_1,\bb t_2)} \sqrt{\var \widetilde{Q}_N(\bb t_1,\bb t_2)},
\end{align*}
where
\begin{align*}
Q_N(\bb t_1,\bb t_2) &= \sum_{i=1}^N d_{\pi_i} [a_{\bb t_1}(i) -a_{\bb t_2}(i)],\\
\widetilde{Q}_N(\bb t_1,\bb t_2) &= \sum_{i=1}^N d_{\pi_i}[a_{\bb t_1}(i) +a_{\bb t_2}(i)]. 
\end{align*}
See that
\begin{align*}
|a_{\bb t_1}(i) -a_{\bb t_2}(i)| &\leq |\cos(\bb t_1^\top \bb g_i) -  \cos(\bb t_2^\top \bb g_i) | + |\sin(\bb t_1^\top \bb g_i) -  \sin(\bb t_2^\top \bb g_i) | \\
&\leq 2 \|\bb t_1 - \bb t_2\| \|\bb g_i\|.
\end{align*}
Similar computations as for $S_N$  give that
\begin{align*}
\var Q_N(\bb t_1,\bb t_2) &= \frac{1}{N-1}\sum_{i=1}^N[a_{\bb t_1}(i) -a_{\bb t_2}(i)]^2 - \frac{N}{N-1}[\overline{a}_{\bb t_1,N} - \overline{a}_{\bb t_2,N}]^2 \\
& \leq   \frac{N}{N-1} 4  \|\bb t_1 - \bb t_2\|^2 \frac{1}{N} \sum_{i=1}^N \| \bb g_i\|^2 < M_2  \|\bb t_1 - \bb t_2\|^2,
\end{align*}
where $M_2>0$. This follows from the fact that $ \frac{1}{N} \sum_{i=1}^N \| \bb g_i\|^2 \to \int_{\R^p} \|\tx\|^2 \mathrm{d} \nu(\tx) <\infty$.  Similarly, 
it follows that
$
\var \widetilde{Q}_N(\bb t_1,\bb t_2) \leq M_3
$
 for some constant $M_3>0$. This implies that \eqref{eq:C2} holds. 
\end{proof}

\begin{proof}[Proof of Theorem~\ref{th2}]
It follows from \eqref{weight} that for any $\tx,\bb{y}\in\R^p$, 
\[
\int_{\R^p} \cos(\bb{t}^\top\tx)\sin(\bb{t}^\top \bb{y}) w(\bb{t})\mathrm{d} \bb{t} = 0.
\]
Therefore, 
\[
D_{n,m} = \int_{\R^p} Z_{n,m}^2(\bb{t}) w(\bb{t})\mathrm{d}\bb{t}
\]
with
\begin{align*}
Z_{n,m} (\bb{t})&= \sqrt{\frac{nm}{n+m}} \left[ \mathrm{Re}(\widehat{\phi}_n)(\bb{t}) +\mathrm{Im}(\widehat{\phi}_n)(\bb{t})  - \mathrm{Re}(\widehat{\phi}_m^{(0)})(\bb{t}) -\mathrm{Im}(\widehat{\phi}_m^{(0)})(\bb{t})   \right]\\
& =  \sqrt{\frac{nm}{n+m}}\Bigg\{ \frac{1}{n}\sum_{i=1}^n \Big[\cos\big(\bb{t}^\top \widehat{\bb G}(\tX_i)\big) +\sin\big(\bb{t}^\top \widehat{\bb G}(\tX_i)\big) \Big]\\
&\quad \quad\quad\quad \quad\quad- \frac{1}{m}\sum_{i=1}^m \Big[\cos\big(\bb{t}^\top\widehat{\bb G}(\tX_i^{(0)})\big) 
+\sin\big(\bb{t}^\top\widehat{\bb G}(\tX_i^{(0)})\big)\Big] \Bigg\}.
\end{align*}
Under the null hypothesis, the vector of ranks $\Big(\widehat{\tG}(\tX_1),\dots,\widehat{\tG}(\tX_n),\widehat{\tG}(\tX_1^{(0)}),\dots, \widehat{\tG}(\tX_m^{(0)})\Big)^\top$ has a uniform distribution over the $N!$ permutations of the grid set $\mathcal{G}_N$, so the distribution of $Z_{n,m}(\bb{t})$ is the same as the distribution of
\[
\sqrt{\frac{nm}{n+m}}\Bigg\{ \frac{1}{n}\sum_{i=1}^n \bigl[\cos(\bb{t}^\top\bb{g}_{\pi_i}) +\sin(\bb{t}^\top \bb{g}_{\pi_i}) \bigr]- \frac{1}{m}\sum_{i=n+1}^{N} \bigl[\cos(\bb{t}^\top\bb{g}_{\pi_i}) +\sin(\bb{t}^\top \bb{g}_{\pi_i})\bigr] \Bigg\}. 
\]
It follows from Lemma~\ref{lem2} that the distribution of $Z_{n,m}(\bb{t})$ is the same as the distribution of
\[
 \sqrt{\frac{n+m}{nm}} \sum_{i=1}^N c_i \bigl[\cos(\bb{t}^\top\bb{g}_{\pi_i})+\sin(\bb{t}^\top\bb{g}_{\pi_i})\bigr]
\]
for $c_i  =m/(n+m)$ for $1\leq i\leq n$, and $c_i = -n/(n+m)$ for $n+1\leq i\leq N$.
Therefore, the distribution of  $Z_{n,m}(\bb{t})$ is the same as the distribution of 
\begin{equation}\label{eq:Z0}
Z_{n,m}^{(0)}(\bb{t}) =  \sqrt{\frac{n+m}{nm}} \sum_{i=1}^N c_{\pi_i} \bigl[\cos(\bb{t}^\top\bb{g}_{i})+\sin(\bb{t}^\top\bb{g}_{i})\bigr] = \sum_{i=1}^N d_{\pi_i} \bigl[\cos(\bb{t}^\top\bb{g}_{i})+\sin(\bb{t}^\top\bb{g}_{i})\bigr] .
\end{equation}
It follows from Lemma~\ref{lem2} that 
\[
\int_F [Z_{n,m}^{(0)}(\bb t)]^2 w(\bb t) \mathrm{d} \bb t \stackrel{D}{\to} \int_F Z^2(\bb t) w(\bb t) \mathrm{d} \bb t 
\]
for a centered Gaussian process $Z$ with the covariance function in \eqref{eq:R} and for any compact set $F\subset \R^p$. 
It also follows from the proof of Lemma~\ref{lem3} that  $\E  [Z_{n,m}^{(0)}(\bb t)]^2\leq M_1$ for a constant $M_1>0$. Since $w$ is integrable on $\R^p$, there exists a compact set $F_{\eps}$ such that
$
\E \int_{\R^p\setminus F_{\eps}} [Z_{n,m}^{(0)}(\bb t)]^2 w(\bb t) \mathrm{d} \bb t <\eps.
$
Analogous arguments apply also to process $Z$, so 
$
\E \int_{\R^p\setminus F_{\eps}} [Z(\bb t)]^2 w(\bb t) \mathrm{d} \bb t <\eps.$
This finishes the proof. 
\end{proof}

\begin{lemma}\label{lem:triangle}
Let $w$ satisfy \eqref{weight} and $K= \int_{\R^p} \|\bb t\| w(\bb t) \mathrm{d} \bb t <\infty$. Then
\[
|C_w(\bb x_1 -\bb y_1) - C_w(\bb x_2 -\bb y_2)| \leq K \left(\|\bb x_1 - \bb x_2\| + \|\bb y_1 -\bb y_2\|\right).
\] 
\end{lemma}
\begin{proof} See that
\[
\left| \cos(a) - \cos(b) \right| = \left|2 \sin\left(\frac{a+b}{2}\right)\sin\left(\frac{a-b}{2}\right)\right| \leq 2 \left|  \sin\left(\frac{a-b}{2}\right)\right| \leq |a-b|.
\]
We get from \eqref{Cw} and from the Cauchy-Schwarz inequality that
\begin{align*}
|C_w(\bb x) - C_w(\bb y)| &\leq \int_{\R^p} |\cos(\bb t^\top \bb x) - \cos(\bb t^\top \bb y)| w(\bb t) \mathrm{d} \bb t \\
& \leq \int_{\R^p} |\bb t^\top (\bb x -\bb y)| w(\bb t) \mathrm{d} \bb t \\
& \leq  \|\bb x -\bb y \|  \int_{\R^p}\|\bb t\| w(\bb t)  \mathrm{d} \bb t  = K   \|\bb x -\bb y \|.
\end{align*}
The assertion then follows from the standard triangular inequality.
\end{proof}

\begin{lemma}\label{lem:as}
Consider the same assumptions as in Theorem~\ref{th3}. 
 Let $\widehat{\bb G}_N$ be the optimal transport of the pooled sample  to $\mathcal{G}_N$.
Then
\[
\frac{1}{n} \sum_{i=1}^n \big\| \widehat{\bb G}_N(\bb X_i) - \bb G(\bb X_i) \big \| \to 0, \quad  \frac{1}{m}\sum_{j=1}^m \big\| \widehat{\bb G}_N(\bb X_i^{(0)}) - \bb G(\bb X_i^{(0)}) \big \| \to 0
\]
with probability 1 as $n \to \infty$. 
\end{lemma}
\begin{proof}
Let $\mu_N$ be the empirical measure on the pooled sample $\{\bb X_1,\dots,\bb X_n, \bb X_1^{(0)},\dots, \bb X_m^{(0)}\}$. Then $\mu_N$ converges weakly to $\mu$ as $n\to \infty$. Furthermore, the assumption of finiteness of the second moments of both $\mathcal{P}_{\bb X}$ and $\mathcal{P}_0$ and the compactness of $S$ imply that $\mu_N$ converges to $\mu$ also in the Wasserstein space $W_2$ \cite[Chapter 5]{santam}. Moreover,  as the mixture $\mu$ is absolutely continuous,  Brenier's theorem \cite[Theorem 2.32]{villani} ensures the existence and uniqueness of the optimal transport $\bb G$. It then follows from the stability of optimal transport plans, \cite[Chapter 1]{santam}, that $\widehat{\bb G}_N$ converges to $\bb G$ in a sense that
\[
\frac{1}{N}\left[ \sum_{i=1}^n \| \widehat{\bb G}_N(\bb X_i) - {\bb G}(\bb X_i) \|^2 +  \sum_{j=1}^m \|\widehat{\bb G}_N(\bb X_i^{(0)}) - {\bb G}(\bb X_i^{(0)}) \|^2 \right] \to 0
\]
with probability 1 as $n\to\infty$.
The desired assertion then follows from the Cauchy-Schwarz inequality. 
\end{proof}

\begin{proof}[Proof of Theorem~\ref{th3}.]
%
Let $\bb G$ be as in Lemma~\ref{lem:as}. It follows from \eqref{sum} that
\[
\frac{D_{n,m}}{n+m} =  \underbrace{\frac{m}{n (n+m)^2} S_1}_{J_1} + \underbrace{\frac{n}{m (n+m)^2} S_2}_{J_2}- 2\cdot \underbrace{\frac{1}{(n+m)^2} \cdot S_3}_{J_3},
\]
where
\begin{align*}
S_1
&= \sum_{j=1}^n\sum_{k=1}^n C_w(\boldsymbol R_{j}-\boldsymbol R_{k}),\\
S_2 &=   \sum_{j=1}^{m}\sum_{k=1}^m C_w(\boldsymbol R^{(0)}_{j}-\boldsymbol R^{(0)}_{k}),  \\ 
S_3 &= \sum_{j=1}^n \sum_{k=1}^{m} C_w(\boldsymbol R_{j}-\boldsymbol R^{(0)}_{k}).
\end{align*}
It follows from Lemma~\ref{lem:triangle} that 
\[
|C_w(\bb R_j - \bb R_k ) - C_w(\bb G(\bb X_i) - \bb G(\bb X_k))| \leq K\left(\| \widehat{\bb G}_N(\bb X_j) - \bb G(\bb X_j) \| + \| \widehat{\bb G}_N(\bb X_k) - \bb G(\bb X_k) \| \right),
\] 
and consequently
\[
S_1 \geq \sum_{j=1}^n\sum_{k=1}^n C_w(\bb G(\bb X_j) - \bb G(\bb X_k)) - 2K\cdot n\sum_{j=1}^n \| \widehat{\bb G}_N(\bb X_j) - \bb G(\bb X_j) \|.   
\]
Hence,
\begin{align*}
\liminf_{n\to \infty} J_1 &\geq \liminf_{n \to \infty} \frac{m}{n (n+m)^2} \sum_{j=1}^n\sum_{k=1}^n C_w(\bb G(\bb X_j) - \bb G(\bb X_k))  \\
& \quad -2K \limsup_{n\to\infty}  \frac{m n}{n (n+m)^2} \sum_{j=1}^n \| \widehat{\bb G}_N(\bb X_j) - \bb G(\bb X_j) \| \\
& = \frac{\lambda}{(1+\lambda)^2} \E C_w(\bb G(\bb X_j) - \bb G(\bb X_k)), 
\end{align*}
with probability 1, due to Lemma~\ref{lem:as}. Similarly, 
\[
S_1 \leq \sum_{j=1}^n\sum_{k=1}^n C_w(\bb G(\bb X_j) - \bb G(\bb X_k)) + 2K\cdot n\sum_{j=1}^n \| \widehat{\bb G}_N(\bb X_j) - \bb G(\bb X_j) \|
\]
and the same arguments lead to the conclusion that $\limsup_{n\to\infty} J_1 \leq \frac{\lambda}{(1+\lambda)^2} \E C_w(\bb G(\bb X_i) - \bb G(\bb X_k))$ with probability 1, which together with the previous implies that 
\[
J_1 \stackrel{\mathsf{P}}{\to} \frac{\lambda}{(1+\lambda)^2} \E C_w(\bb G(\bb X_i) - \bb G(\bb X_k))
\]
as $n \to \infty$. 
Analogously, one can show that as $n\to \infty$,
\[
 J_2 \stackrel{\mathsf{P}}{\to}  \frac{\lambda}{(1+\lambda)^2} \E C_w(\bb G(\bb X_j^{0)}) - \bb G(\bb X_k^{(0)}))
\]
 and
\[
J_3 \stackrel{\mathsf{P}}{\to} \frac{\lambda}{(1+\lambda)^2} \E C_w(\bb G(\bb X_j^{0)}) - \bb G(\bb X_k)).
\]
Let $\bb U_i = \bb G(\bb X_i)$, $\bb V_i = \bb G(\bb X_i^{(0)})$, $i=1,2$. Then we get from the previous that as $n\to\infty$,
\begin{align}
\frac{D_{n,m}}{n+m} &\stackrel{\mathsf{P}}{\to} \frac{\lambda}{(1+\lambda)^2} 
\left[\E C_w(\bb U_1- \bb U_2)+  \E C_w(\bb V_1 - \bb V_2) - 2  \E C_w(\bb U_1 - \bb V_2)\right] \notag \\
& =  \frac{\lambda}{(1+\lambda)^2}  \int_{\R^p}\left| \exp(\ii\bb t^\top \bb U_1) - \exp(\ii \bb t^\top \bb V_1)\right|^2 w(\bb t) \mathrm{d} \bb t,\notag \\
 & =  \frac{\lambda}{(1+\lambda)^2}  \int_{\R^p}\left|\varphi_{\bb U}(\bb t) - \varphi_{\bb V}(\bb t)\right|^2 w(\bb t) \mathrm{d} \bb t,\label{eq:difUV}
\end{align}
where $\varphi_{\bb U}$ and $\varphi_{\bb V}$ are the characteristic functions of $\bb U_1$ and $\bb V_1$, respectively.

Under the stated assumptions, both $\mu$ and $\nu$ are absolutely continuous with finite second moments. Hence, if $\bb{H}$ is the optimal transport that pushes $\nu$ to $\mu$, then $\bb H(\bb G(\tx)) = \tx$ $\mu$-a.s., see \cite[Corollary 6.6.]{galichon} and 
since $\mu$ is the mixture of $\mathcal{P}_{\bb X}$ and $\mathcal{P}_0$, the equality holds also $\mathcal{P}_{\bb X}$ a.s. and  $\mathcal{P}_{0}$ a.s. If $\bb U_1 = \bb G(\bb X_1)$ has the same distribution as $\bb V_1 = \bb G(\bb X_1^{(0)})$, then $\bb X_1 = \bb H(\bb G(\bb X_1)) = \bb H(\bb U_1))$ needs to have the same distribution as $\bb X_1^{(0)} = \bb H(\bb G(\bb X_1^{(0)})) = \bb H(\bb V_1)$. Consequently, if $\mathcal{P}_{\bb X} \ne \mathcal{P}_0$, then $\varphi_{\bb U} \ne \varphi_{\bb V}$ and the integral in \eqref{eq:difUV} is positive for all weight functions $w$ positive on a neighborhood of $\bb 0$.

\end{proof}

\section{Additional results to the practical part} \label{app:d}

Algorithm~\ref{alg0} describes a procedure for calculation of the asymptotic critical values from Theorem~\ref{th2}. Table \ref{tab:crit:est} contains bootstrap  critical values $\widetilde{c}_{n,m,\alpha}$ calculated for samples generated from $\mathsf{N}_2(\bm{0},\bm{I})$ and weighting by $C_{\alpha,\gamma}$ from \eqref{eq:Cw} with $\gamma=2$ and various values of $a>0$.

\begin{algorithm}[bhtp]
\caption{Approximate asymptotic critical values $c_{\alpha}$.}
\begin{algorithmic}[1]
\Require Parameter  $K>0$ and large integers $M,G,B\in\mathbb{N}$.  
\State  Choose a partition $\{\mathcal{V}_i\}_{i=1}^G$ of the set $[-K,K]^p$ and 
a corresponding grid  $\{\bb t_i\}_{i=1}^G$ such that $\bb t_i \in \mathcal{V}_i$.
\State  Choose a grid of $M$ points 
in the support of $\nu$ such that the corresponding empirical  measure $\nu_M$  
reasonably approximates $\nu$.
\State  Approximate the covariances $R(\bb t_i, \bb t_j)$, $i,j=1,\dots,G$, in \eqref{eq:R} 
by the covariances computed with $\tY$ replaced by $\tY_M$
with the discrete  distribution $\nu_M$ (so the covariances can be computed as finite sums). 
\For{$b=1$ in $B$}
\State Generate a realization of the process $Z$ in points $\bb t_i$, $i=1,\dots,G$, as a centered multivariate normal random vector $\big(Z^b(\bb t_1),\dots, Z^b(\bb t_G)\big)^\top$ with covariances obtained in step 3.
\State Calculate the limit variable from \eqref{limit} by a numerical integration using the integration grid from step 1, that is
$
D^b = \sum_{i=1}^G [Z^b(\bb t_i)]^2 w(\bb t_i) \mathrm{vol}(\mathcal{V}_i),
$
where $\mathrm{vol}(\cdot)$ stands for the volume of a set. 
\EndFor
\State  Calculate the critical value $c_{\alpha}$ as the corresponding $(1-\alpha)$-sample quantile of $D^1,\dots,D^B$. 
\Ensure An approximation to the asymptotic critical value $c_{\alpha}$.
\end{algorithmic}
\label{alg0}
\end{algorithm}

\begin{table}[htbp]
\scriptsize
\caption{Bootstrap critical values $\widetilde{c}_{n,m,\alpha}$ for $p=2$ 
computed for samples generated from $\mathsf{N}_2(\bm{0},\bm{I})$. 
The reference sample is generated either randomly (R) or computed as a set of transformed grid points from \eqref{eq:xtilde} (G).  Weighting by $C_{\alpha,\gamma}$ from \eqref{eq:Cw} with $\gamma=2$ and various values of $a>0$, sample sizes $n$ and $m$, $\alpha=0.05$, 10\,000 bootstrap replications.}
\label{tab:crit:est}
\begin{tabular}{rr|rrrrrr} 
\toprule
  & \multicolumn{1}{c}{}& \multicolumn{3}{c}{R}&\multicolumn{3}{c}{G} \\ 
\cmidrule(r){3-5}\cmidrule(r){6-8}
  \multicolumn{1}{c}{$m$} & \multicolumn{1}{c}{$a$} & $n=20$&$n=50$&$n=80$&$n=20$&$n=50$&$n=80$ \\ 
 \cmidrule(r){1-2}\cmidrule(r){3-3}\cmidrule(r){4-4}\cmidrule(r){5-5}\cmidrule(r){6-6}\cmidrule(r){7-7}\cmidrule(r){8-8}
 &0.5&0.0352&0.0577&0.1218&0.0136&0.0128&0.0130 \\ 
 &1.0&0.1246&0.1939&0.3563&0.0981&0.0903&0.0835 \\ 
 &1.5&0.4687&0.5574&0.6280&0.2874&0.2638&0.2404 \\ 
 200&2.0&0.4433&0.5462&0.8550&0.5020&0.4516&0.4060 \\ 
 &2.5&0.9083&0.9740&1.0123&0.6722&0.5866&0.5379 \\ 
 &3.0&1.0193&1.0784&1.1115&0.7900&0.6895&0.6393 \\ 
 &3.5&1.1131&1.1462&1.1707&0.8792&0.7592&0.6877 \\ 
 \cmidrule(r){2-8}
 &0.5&0.0285&0.0329&0.0674&0.0117&0.0114&0.0108 \\ 
 &1.0&0.1493&0.1228&0.2346&0.1007&0.0936&0.0905 \\ 
 &1.5&0.3961&0.4584&0.4934&0.3137&0.2854&0.2737 \\ 
 500&2.0&0.6805&0.4325&0.7640&0.5509&0.4913&0.4661 \\ 
 &2.5&0.8784&0.9206&0.9315&0.7334&0.6523&0.6091 \\ 
 &3.0&0.9984&1.0393&1.0542&0.8637&0.7529&0.7195 \\ 
 &3.5&1.0814&1.1014&1.1246&0.9274&0.8337&0.7805 \\ 
 \cmidrule(r){2-8}
 &0.5&0.0193&0.0314&0.0419&0.0106&0.0111&0.0105 \\ 
 &1.0&0.1290&0.1586&0.1814&0.1014&0.1021&0.0949 \\ 
 &1.5&0.3827&0.4215&0.4442&0.3187&0.3122&0.2942 \\ 
 1000&2.0&0.6478&0.6940&0.7122&0.5664&0.5345&0.4985 \\ 
 &2.5&0.8568&0.8837&0.9070&0.7586&0.6984&0.6618 \\ 
 &3.0&0.9822&1.0046&1.0344&0.8837&0.8221&0.7792 \\ 
 &3.5&1.0755&1.0759&1.0882&0.9694&0.8957&0.8236 \\ 
\botrule
 \end{tabular}
\end{table}

Tables~\ref{tab:single:p=2} and \ref{tab:single:p=4} present results for the simple null hypothesis $\mathcal{H}_0: F_{\tX} = F_0$, where $F_0$ is DF of $\mathsf{N}_p(\bb 0, \bb I)$ for dimensions $p=2$ and $p=4$, respectively. They lead to analogous conclusions as Table~\ref{tab:single:p=3}, presented in the main text.

Table~\ref{tab:normality:p=34} complements results from Table~\ref{tab:warp} for testing the composite hypothesis of normality for dimensions $p\in\{3,4\}$. The test based on $\widetilde{D}_{n,m}^\star$ is compared to the same battery of tests from the R library \texttt{MVN} \citep{MVN}.

\begin{table}[htbp]
\scriptsize
\caption{Empirical level and power (in \%) for a goodness-of-fit test for bivariate normal $\mathsf{N}_2(\boldsymbol{0},\boldsymbol{I})$ distribution. 
$\mathsf{U}_{A}$ denotes the uniform distribution on the set $A$, $\mathsf{T}_2(\bm{0},\boldsymbol{I},3)$ is the bivariate $t$ distribution with 3 degrees of freedom, location $\bm{0}$ and identity scale matrix.
The values are computed from 1\,000 simulations for the nominal significance level $\alpha=0.05$.}
\label{tab:single:p=2}
\begin{tabular}{llr|rrr|rrr} 
\toprule
 & & & \multicolumn{3}{c}{$\mathcal{G}_N^R$}&\multicolumn{3}{c}{$\mathcal{G}_N^S$} \\ 
\cmidrule(r){4-6}\cmidrule(r){7-9}
 &\multicolumn{1}{c}{$m$} & \multicolumn{1}{c|}{$a$} & $n=20$&$n=50$&$n=80$&$n=20$&$n=50$&$n=80$ \\ 
\midrule
&&0.5&5.5&4.8&5.8&5.0&5.6&5.8 \\ 
 &&1.0&6.4&5.5&6.4&4.4&5.1&4.9 \\ 
 &200&2.0&4.9&3.7&3.7&5.0&5.6&5.1 \\ 
 &&3.0&4.9&4.7&3.4&5.1&4.5&5.2 \\ 
 &&4.0&5.6&5.8&6.3&5.9&4.4&3.6 \\ 
\cmidrule{3-3}\cmidrule{4-9} 
 \raisebox{1.5ex}[0pt]{$\mathsf{N}_2(\boldsymbol{0},\boldsymbol{I})$}&&0.5&5.3&4.4&5.4&3.9&4.6&5.1 \\ 
 &&1.0&5.0&4.4&4.6&3.2&5.9&4.3 \\ 
 &500&2.0&4.9&5.2&4.3&4.3&5.1&5.2 \\ 
 &&3.0&5.2&6.6&6.4&3.9&4.9&4.8 \\ 
 &&4.0&6.6&4.5&4.9&3.7&5.3&5.2 \\ 
\midrule
 &&0.5&0.6&2.4&4.4&0.5&13.8&64.3 \\ 
 &&1.0&3.1&11.7&27.9&40.0&99.5&100.0\\ 
 &200&2.0&22.4&74.1&91.2&86.1&100.0&100.0 \\ 
 &&3.0&38.1&82.8&94.8&79.1&99.9&100.0 \\ 
 &&4.0&41.1&85.8&96.3&72.4&99.5&100.0 \\ 
\cmidrule{3-3}\cmidrule{4-9}
 \raisebox{1.5ex}[0pt]{$\mathsf{U}_{(-1,1)^2}$}&&0.5&0.5&1.9&3.6&0.2&5.4&81.5 \\ 
 &&1.0&2.5&11.4&38.6&53.4&100.0&100.0 \\ 
 &500&2.0&24.7&85.0&98.9&92.5&100.0&100.0 \\ 
 &&3.0&41.2&93.5&99.8&90.5&100.0&100.0 \\ 
 &&4.0&49.0&95.8&100.0&79.4&100.0&100.0 \\ 
\midrule
 &&0.5&9.2&10.9&8.2&13.2&15.8&18.2 \\ 
 &&1.0&10.2&13.5&16.7&25.5&52.6&68.1 \\ 
 &200&2.0&21.0&42.5&60.0&46.0&80.1&92.7 \\ 
 &&3.0&24.4&55.5&76.1&48.7&83.0&94.2 \\ 
 &&4.0&29.2&58.8&77.1&39.9&81.0&92.7 \\ 
\cmidrule{3-3}\cmidrule{4-9}
 \raisebox{1.5ex}[0pt]{$\mathsf{U}_{(-2,2)^2}$}&&0.5&10.9&11.1&11.5&15.4&16.5&20.3 \\ 
 &&1.0&13.8&18.3&26.0&31.2&58.6&81.0 \\ 
 &500&2.0&22.7&49.7&75.9&50.6&88.2&97.6 \\ 
 &&3.0&29.4&64.3&85.8&51.4&85.5&98.3 \\ 
 &&4.0&28.0&66.7&87.1&46.8&86.4&96.3 \\ 
\midrule
 &&0.5&7.1&7.3&5.3&8.0&9.2&9.7 \\ 
 &&1.0&6.9&6.4&9.5&9.2&15.4&18.0 \\ 
 &200&2.0&7.4&6.3&9.5&8.9&15.6&22.1 \\ 
 &&3.0&8.1&7.8&10.9&8.1&12.2&16.3 \\ 
 &&4.0&6.8&8.1&10.7&8.3&11.0&13.5 \\ 
\cmidrule{3-3}\cmidrule{4-9}
 \raisebox{1.5ex}[0pt]{$\mathsf{T}_2(\bm{0},\boldsymbol{I},3)$}&&0.5&8.5&7.0&7.7&11.9&10.5&12.8 \\ 
 &&1.0&9.3&9.6&8.5&12.3&18.5&23.5 \\ 
 &500&2.0&6.3&9.2&12.9&10.5&17.1&26.3 \\ 
 &&3.0&7.1&10.7&15.0&9.0&13.0&16.9 \\ 
 &&4.0&6.5&9.8&12.8&6.4&11.4&15.8 \\ 
\botrule
 \end{tabular}
\end{table}

\begin{table}[htbp]
\scriptsize
\caption{Empirical level and power (in \%) for a goodness-of-fit test for  $\mathsf{N}_4(\boldsymbol{0},\boldsymbol{I})$ distribution. 
$\mathsf{U}_{A}$ denotes the uniform distribution on the set $A$, $\mathsf{T}_4(\bm{0},\boldsymbol{I},3)$ is $4$-variate $t$ distribution with 3 degrees of freedom, location $\bm{0}$ and identity scale matrix.
The values are computed from 1\,000 simulations for the nominal significance level $\alpha=0.05$.}
\label{tab:single:p=4}
\begin{tabular}{llr|rrr|rrr} 
\toprule
 & & & \multicolumn{3}{c}{$\mathcal{G}_N^R$}&\multicolumn{3}{c}{$\mathcal{G}_N^S$} \\ 
\cmidrule(r){4-6}\cmidrule(r){7-9}
 &\multicolumn{1}{c}{$m$} & \multicolumn{1}{c|}{$a$} & $n=20$&$n=50$&$n=80$&$n=20$&$n=50$&$n=80$ \\ 
\midrule
 &&1.0&5.5&5.2&5.9&5.3&4.5&4.4 \\
 &&2.0&4.4&5.3&5.1&5.3&5.2&6.3 \\
 &200&3.0&5.6&5.7&4.9&4.7&4.7&5.1 \\
 &&4.0&4.9&5.3&4.6&5.2&4.4&6.6 \\
  \cmidrule{3-3}\cmidrule{4-9} 
\raisebox{1.5ex}[0pt]{$\mathsf{N}_4(\boldsymbol{0},\boldsymbol{I})$}&&1.0&5.7&4.3&4.2&5.9&4.4&4.7 \\
 &&2.0&5.3&5.2&5.0&4.4&6.4&4.5 \\
 &500&3.0&4.2&5.6&4.6&5.1&5.6&6.1 \\
 &&4.0&3.7&3.9&4.4&4.3&5.8&4.8 \\
  \cmidrule{1-3}\cmidrule{4-9} 
 &&1.0&11.5&15.8&22.0&43.1&69.4&83.6 \\
 &&2.0&17.8&41.1&68.0&67.9&94.7&99.1 \\
 &200&3.0&14.5&45.0&68.2&70.0&95.5&99.5 \\
 &&4.0&11.4&33.6&64.4&66.6&94.9&99.3 \\
  \cmidrule{3-3}\cmidrule{4-9} 
 \raisebox{1.5ex}[0pt]{$\mathsf{U}_{(-2,2)^4}$}&&1.0&15.0&26.9&42.2&48.9&83.1&96.8 \\
 &&2.0&22.0&60.0&87.2&74.0&98.0&99.8 \\
 &500&3.0&20.4&62.3&87.1&75.3&97.5&99.8 \\
 &&4.0&15.1&48.1&78.4&70.1&98.2&99.9 \\
  \cmidrule{1-3}\cmidrule{4-9} 
 &&1.0&6.2&5.5&6.8&14.5&22.8&28.3 \\
 &&2.0&4.3&8.5&7.5&16.1&22.2&28.7 \\
 &200&3.0&5.3&5.5&6.8&9.0&10.9&13.3 \\
 &&4.0&5.2&5.4&6.3&7.7&6.0&7.2 \\
  \cmidrule{3-3}\cmidrule{4-9} 
\raisebox{1.5ex}[0pt]{$\mathsf{T}_4(\bm{0},\boldsymbol{I},3)$}&&1.0&6.7&7.0&9.1&17.3&30.8&37.4 \\
 &&2.0&4.8&7.1&7.7&17.7&22.4&32.7 \\
 &500&3.0&6.2&7.0&7.2&11.5&11.8&19.6 \\
 &&4.0&4.5&6.9&8.3&6.9&7.1&7.9 \\
\botrule
\end{tabular}
\end{table}

\begin{table}[htbp]
\scriptsize
\caption{Empirical level and power (in \%) for goodness-of-fit tests for  $\mathsf{N}_p(\boldsymbol{0},\boldsymbol{I})$ distribution, dimension $p\in\{3,4\}$. 
$\mathsf{U}_{A}$ denotes the uniform distribution on the set $A$, $\mathsf{T}_p(\bm{0},\boldsymbol{I},3)$ is $p$-variate $t$ distribution with 3 degrees of freedom, location $\bm{0}$ and identity scale matrix.
D-S denotes a test that uses the two-sample energy test statistic from \citet{deb2023}, while
the other reference tests are  computed using R library \texttt{MVN}.  
The values are computed from 1\,000 simulations for the nominal significance level $\alpha=0.05$.}
\label{tab:normality:p=34}
\begin{tabular}{rc|rrr|rrr|rrr} 
\toprule 
  & &\multicolumn{3}{c}{$\mathsf{N}_p (\boldsymbol{0},\boldsymbol I)$}&\multicolumn{3}{c}{$\mathsf{U}_{(-1,1)^p}$}&\multicolumn{3}{c}{$\mathsf{T}_p(\bm{0},\boldsymbol{I},3)$} \\[1ex] 
  && \multicolumn{9}{c}{Sample size $n$}\\
   \cmidrule(r){1-11} 
  & $a$ & 20&50&80&20&50&80&20&50&80 \\ 
  \cmidrule(r){1-11} 
$p=3$ &0.5&4.1&3.4&4.0&2.9&0.7&0.5&37.0&89.0&97.8 \\ 
 &1.0&3.2&3.7&6.6&0.3&0.1&0.7&60.2&95.0&99.5 \\ 
 &2.0&6.8&4.4&5.0&11.0&49.9&85.5&63.9&97.0&99.8 \\ 
 &2.5&4.7&3.2&6.5&20.7&68.4&93.3&56.6&94.5&99.5 \\ 
 &3.0&4.0&4.3&5.5&23.1&70.3&94.0&55.3&95.8&99.3 \\ 
 &4.0&4.2&3.9&3.2&25.5&68.8&90.7&47.3&94.9&99.3 \\ 
 &5.0&5.3&5.2&5.9&28.2&62.8&87.7&48.6&91.1&98.9 \\ 
 \cmidrule(r){1-2}\cmidrule(r){3-11} 
 & \multicolumn{1}{l|}{D-S} &4.6&4.2&5.0&8.1&41.2&69.6&60.7&95.7&99.8 \\ 
 \cmidrule(r){1-2}\cmidrule(r){3-11} 
 & \multicolumn{1}{l|}{H-Z}   & 4.6 & 5.1 & 4.6 & 13.8 & 61.3 & 92.4 & 50.7 & 90.3 & 98.5  \\
&\multicolumn{1}{l|}{M3}     & 1.9 & 3.9  &  4.7   &  0.0 &  0.0 &  0.0 &  57.4    & 90.4     & 96.0   \\
&\multicolumn{1}{l|}{M4}     & 0.1 & 1.7 &  3.3   &  0.7 & 75.4 &  98.8 &  39.7    & 95.2     & 99.7   \\
 &\multicolumn{1}{l|}{Royston}& 7.3 & 7.9 & 6.8 & 37.0 & 99.2 & 100.0& 63.2 & 93.1 & 98.8  \\
 &\multicolumn{1}{l|}{D-H}    & 5.2 & 5.1 & 5.6 &  6.7 & 74.7 & 99.2 & 42.2 & 89.7 & 98.1  \\  
 &\multicolumn{1}{l|}{energy} & 6.3 & 6.0 & 5.2 &  7.1 & 42.6 & 82.7 & 63.0 & 95.4 & 99.4  \\
 \midrule
$p=4$ &0.5&5.3&4.7&4.0&2.8&1.8&3.4&21.1&75.7&93.6 \\ 
 &1.0&4.6&4.7&4.5&1.3&0.1&0.0&45.1&91.8&99.7 \\ 
 &2.0&4.3&5.7&5.0&0.7&1.7&9.2&62.3&97.3&99.9 \\ 
 &2.5&4.4&5.6&6.0&3.2&16.2&48.3&66.8&98.6&99.5 \\ 
 &3.0&4.1&3.6&5.4&8.1&27.6&63.6&67.2&97.8&100.0 \\ 
 &4.0&4.6&3.4&4.1&15.3&44.8&71.2&63.6&96.3&99.8 \\ 
 &5.0&3.5&4.2&6.2&26.3&45.6&70.5&63.9&98.5&99.9 \\ 
 \cmidrule(r){1-2}\cmidrule(r){3-11} 
& \multicolumn{1}{l|}{D-S}&2.9&3.8&6.4&1.6&0.9&12.5&58.0&96.7&99.7 \\
\cmidrule(r){1-2}\cmidrule(r){3-11} 
& \multicolumn{1}{l|}{H-Z}   & 5.7 & 4.3 & 5.5 & 11.0 & 55.9 & 90.2 &56.6 & 95.4&  99.6\\
&\multicolumn{1}{l|}{M3}     & 1.5 & 4.2 & 4.8 &  0.0 &  0.0 &  0.0 & 62.5 & 96.3 & 99.1 \\
&\multicolumn{1}{l|}{M4}     & 0.2 & 1.6 & 2.8 &  4.8 & 83.2 &  99.6 & 37.5 & 97.6 & 100.0 \\
 &\multicolumn{1}{l|}{Royston}& 7.5 & 7.6 & 4.9 & 46.5 & 99.8 & 100.0& 70.3& 96.5& 99.7 \\
 &\multicolumn{1}{l|}{D-H}    & 4.8 & 5.8 & 4.4 &  5.0 & 83.1 & 100.0& 37.6& 94.4& 99.3 \\  
 &\multicolumn{1}{l|}{energy} & 5.8 & 7.3 & 4.8 &  4.0 & 29.4 & 64.7 & 70.5& 98.3& 100.0 \\
  \botrule
 \end{tabular}
\end{table}


Figure~\ref{fig:CAS-hist} shows histograms of the income and expenditure data. The red dashed line corresponds to the fitted density of a normal distribution, while the blue solid line is the density of the fitted lognormal distribution.  The plots suggest that the marginal distribution of income exhibits positive skewness, 
while the skewness of expenditures is comparatively modest.


\begin{figure}
    \centering
    \includegraphics[width=0.4\textwidth]{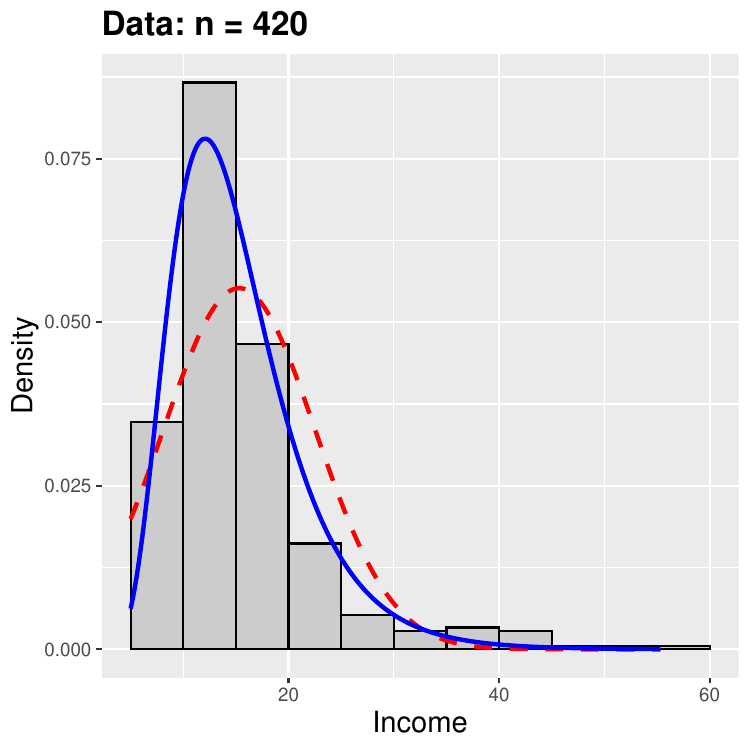}
     \includegraphics[width=0.4\textwidth]{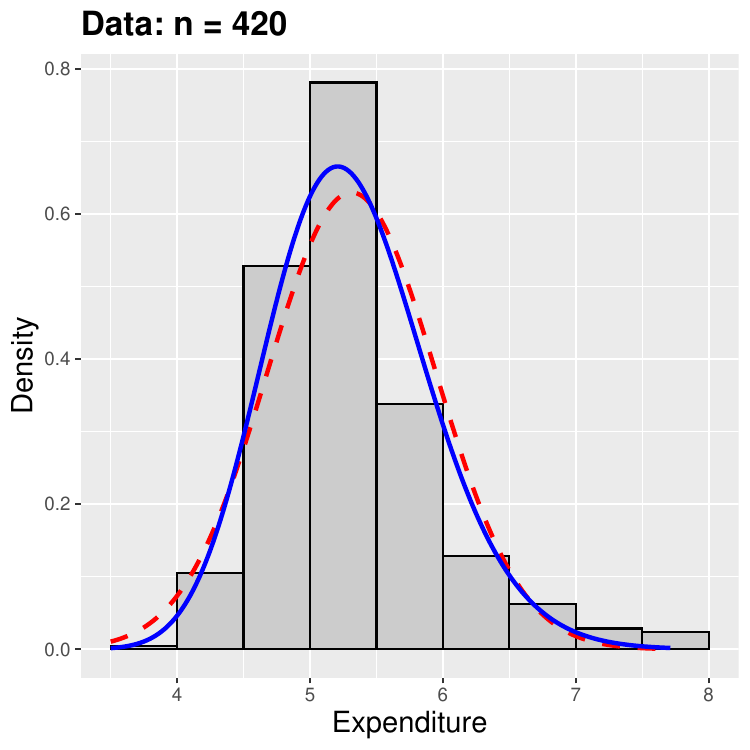}
        \includegraphics[width=0.4\textwidth]{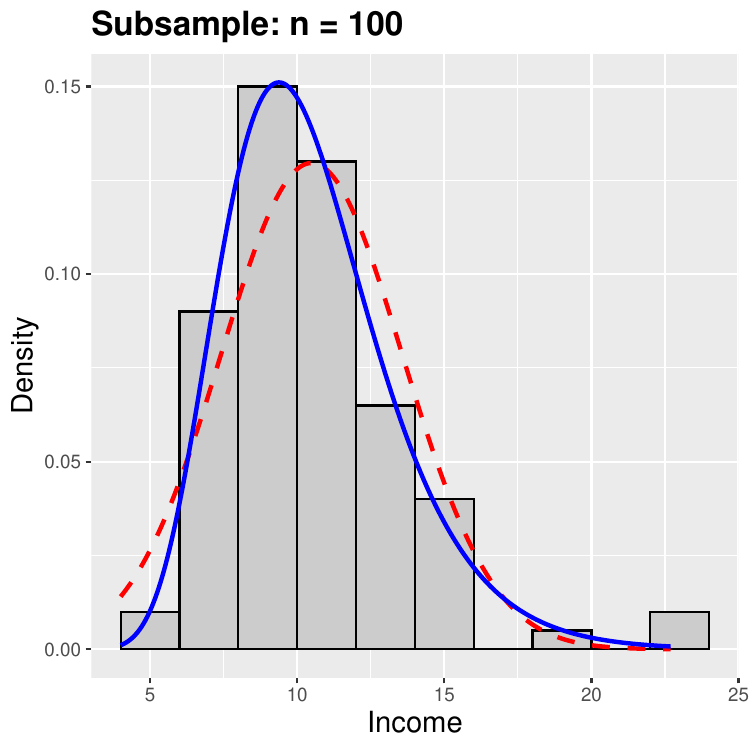}
     \includegraphics[width=0.4\textwidth]{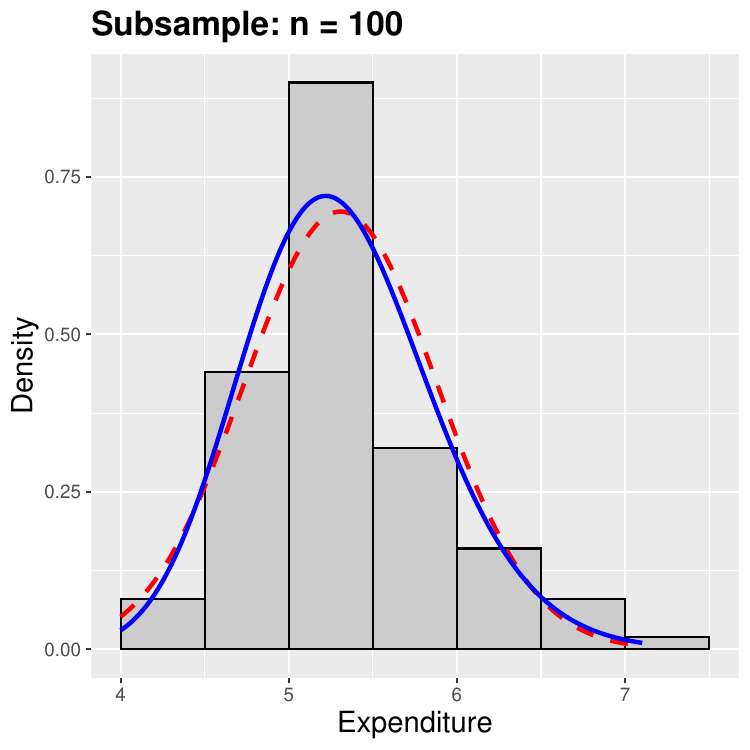}
    \caption{Histograms of income and expenditure data, together with the density of the fitted normal distribution  (red dashed line) and the density of the fitted lognormal distribution (blue solid line).  The entire data set (top panels) and the subsample of the first 100 observations (bottom panels).  
    }
    \label{fig:CAS-hist}
\end{figure}

%
%
%
%

%

\end{document}